\documentclass{amsart}
\usepackage{amsmath}
\usepackage[colorlinks=true]{hyperref}
\hypersetup{urlcolor=blue, citecolor=red}
\newtheorem{theorem}{Theorem}[section]
\newtheorem{corollary}[theorem]{Corollary}

\newtheorem{lemma}[theorem]{Lemma}
\newtheorem{proposition}[theorem]{Proposition}

\theoremstyle{definition}
\newtheorem{definition}[theorem]{Definition}
\newtheorem{remark}{Remark}

\title[Regularisation errors and variable Hilbert scales]{Errors of
regularisation under range inclusions using variable Hilbert scales}
\author[Markus Hegland and Bernd Hofmann]{}
\email{markus.hegland@anu.edu.au}
\email{bernd.hofmann@mathematik.tu-chemnitz.de}

\date{May 20, 2010}

\newcommand{\range}{\operatorname{range}}
\newcommand{\spec}{\operatorname{spec}}
\newcommand{\R}{\mathbb{R}}
\newcommand{\I}{\mathcal{I}}
\newtheorem{example}{Example}

\begin{document}
  \maketitle

\centerline{\scshape Markus Hegland }
\medskip
{\footnotesize
 \centerline{Centre for Mathematics and its Applications}
   \centerline{The Australian National University}
   \centerline{ Canberra ACT, 0200, Australia}
} 

\medskip

\centerline{\scshape Bernd Hofmann}
\medskip
{\footnotesize
 \centerline{Department of Mathematics}
   \centerline{Chemnitz University of Technology}
   \centerline{09107 Chemnitz, Germany.}
}

\bigskip

\begin{abstract}
  Based on the variable Hilbert scale interpolation inequality bounds for the
  error of regularisation methods are derived under range inclusions. In this
  context, new formulae for the modulus of continuity of the inverse of bounded
  operators with non-closed range are given. Even if one can show the
  equivalence of this approach to the version used previously in the literature,
  the new formulae and corresponding conditions are simpler than the former
  ones. Several examples from image processing and spectral enhancement
  illustrate how the new error bounds can be applied.
\end{abstract}

  \section{Introduction}

Let $X$ and $Y$ be infinite dimensional separable Hilbert spaces with norms
$\|\cdot\|$ and scalar products $(\cdot,\cdot)$. We study linear inverse
problems in form of ill-posed operator equations
\begin{equation} \label{opeq}
Af = g, \qquad f \in X,\; g \in Y,
\end{equation}
characterised by an injective bounded linear forward operator $A: X \to Y$ for
which the range $\range(A)$ is a non-closed subset of $Y$. Then equation
(\ref{opeq}) is unstable in the sense that the inverse operator $A^{-1}:
\range(A) \subseteq Y \to X$ is unbounded and hence the use of perturbed data
$g^\delta$ instead of the exact right-hand side $g$ with
\begin {equation} \label{noise}
\|g-g^\delta\| \le \delta
\end{equation}
and noise level $\delta>0$ may lead to arbitrarily large errors in the solution
of (\ref{opeq}) even if the noise level is extremely small. As a consequence of
this ill-posedness phenomenon regularisation methods are required for the stable
approximate solution of the inverse problem. Their basic idea consists in
finding approximations to the exact solution $f$ in form of solutions
$f_\alpha=f_\alpha(g^\delta)$ to stable auxiliary problems neighbouring
(\ref{opeq}). Those solutions are obtained by using the noisy data $g^\delta$.
The degree of neighbourhood of the exploited auxiliary problems is controlled by
a regularisation parameter $\alpha>0$. In this context, small $\alpha$ express
closeness to (\ref{opeq}) in combination with a low level of stability, whereas
larger $\alpha$ ensure better stability, however combined with a low level of
approximation. For the success of any regularisation method an appropriate
trade-off between stability and approximation has to be aspired when choosing
the regularisation parameter.

As already outlined and summarised in the monograph \cite{EHN96} by {\sc Engl,
Hanke} and {\sc Neubauer} a successful way for doing regularisation for linear
ill-posed problems in Hilbert spaces including convergence and convergence rates
of constructed methods requires some knowledge on the impact of smoothness on
the regularised solutions. Smoothness should be understood there in a very
generalised sense as both solution smoothness and smoothing properties of the
forward operator. In \cite{EHN96} such smoothness fitting is focused on H\"older
type source conditions yielding H\"older type convergence rates when the
regularisation method has a sufficiently high level of qualification. This
theory is closely connected with associated (classical) Hilbert scales, where we
refer to the seminal paper by {\sc Natterer} \cite{Nat84}. An extension of that
theory to generalised source conditions implying also more general convergence
rates was performed rather independently by two different approaches. The first
approach initiated and established by {\sc Hegland} (see \cite{Heg92,Heg95})
introduced variable Hilbert scales with positive index functions the behaviour
of which is in particular of interest for large arguments covering the spectrum
of an injective unbounded linear operator with bounded inverse. Results in this
approach are based on interpolation inequalities. An alternative second approach
was developed and published by {\sc Math\'e} and {\sc Pereverzev} (see
\cite{MatPer03b,MatPer03a,MatTau06}) and complimented by ideas of {\sc Hofmann}
and other co-workers (see, e.g., \cite{HofMAt07,HofMS08,HofMW09}). This
approach, in principle, also exploits variable Hilbert scales, but the index
functions occurring there are more specific and their behaviour is of interest
just for small positive arguments covering on the spectrum of $A^*A$. The index
functions of the second approach are monotonically increasing and tend to zero
as the positive arguments tend to zero. Here we call them rate functions
abbreviated by over-lined small Greek letters, because they also express the
convergence rate of approximate solutions. One of the main goals of this paper
is to compare both approaches, their results, required conditions and their
natural interplay. Moreover, some more consequences and new convergence rates
results of {\sc Hegland}'s approach shall be formulated and proven in the
sequel.

In our study we use {\sl variable Hilbert scales} and corresponding {\sl
interpolation inequalities} in order to obtain bounds from above for the error
$\|f-f_\alpha\|$ of regularised solutions under conditions imposed on $f$. As is
well-known the convergence of regularised solutions $f_\alpha \to f$ in $X$,
even if $\delta \to +0$ and $\alpha=\alpha(\delta,g^\delta) \to +0$ is chosen in
an appropriate manner, can be arbitrarily slow for solutions $f$ which are in
some sense non-smooth with respect to the forward operator $A$. It is of
essential interest in regularisation theory to obtain convergence rates
\begin{equation} \label{etarates}
\|f-f_\alpha\|=\mathcal{O}(\bar{\eta}(\delta)) \qquad \mbox{as} \qquad \delta
\to +0
\end{equation}
with appropriate  rate functions $\bar{\eta}$.

For a non-closed range of $A$ convergence rates require that \emph{general
source conditions} are satisfied which attain in the standard case the form
\begin{equation} \label{scclassic}
f=\bar{\psi}(A^*A) v
\end{equation}
with source element $v \in X$ and with some rate function $\bar{\psi}$ defined
at least on the interval $(0,\|A\|^2]$ covering the spectrum of $A^*A$. Here,
$\bar{\psi}(A^*A)$ is well-defined as an injective bounded positive self-adjoint
linear operator by spectral calculus (see, e.g., \cite[Section~2.3]{EHN96}). If
the regularisation method has a sufficiently high {\sl qualification} (see,
e.g., \cite[Definition~2.6]{HofMAt07}), then an asymptotically fast decay of
$\bar \psi(t) \to 0$ as $t \to +0$ corresponds with high order convergence rates
(\ref{etarates}) of regularised solutions and vice versa. Note that the tool of
general source conditions with rate functions $\bar \psi$ different from
monomials was early applied to regularisation theory by {\sc Tautenhahn} in
\cite{Taut96}. Later {\sc Hohage} (see~\cite{Hohage1,Hohage2}) studied in detail
the case of logarithmic functions $\bar \psi$ in (\ref{scclassic}) and
corresponding convergence rates.

Using a non-standard source condition
\begin{equation} \label{Gsmooth}
f=G w
\end{equation}
with source element $w \in X$, the priori information about the smoothness of
the solution $f$ can be expressed by an injective bounded self-adjoint positive
definite linear operator $G: X \to X$, where $\range(G)$ is a non-closed subset
of $X$. In most cases the character of the operator $G$ is rather independent of
the forward operator $A$ and hence $G$ need not be a function of $A^*A$.
However, in order to make conclusions to convergence rates of regularised
solutions the assumption $f \in \range(G)$ of (\ref{Gsmooth}) and the assumption
$f \in \range(\bar{\psi}(A^*A))$ of (\ref{scclassic}) have to be
connected anyway. In the framework of {\sc B\"ottcher} et al.~\cite{Boetal06}
conditions for such connections and their interplay are discussed in a
comprehensive manner. In \cite{HofMAt07} they are called \emph{link conditions}.
Besides the simplest link type $G=\bar{\varphi}(A^*A)$, which is typical for
commuting $G$ and $A^*A$, one of the most convincing class of link conditions
represent \emph{range inclusions} introduced with {\sc Yamamoto} in
\cite{HofYam05} to regularisation theory. The favourite form of such inclusion
is
\begin{equation} \label{Grange}
\range(G) \subseteq \range(\bar{\psi}(A^*A))
\end{equation}
with some rate function $\bar \psi$. Evidently, under (\ref{Gsmooth}) a range
inclusion (\ref{Grange}) immediately implies a source condition
(\ref{scclassic}). The higher the rate expressed by $\bar \psi$ is, i.e. the
faster the decay $\bar \psi(t) \to 0$ as $t \to +0$ goes on, the smaller the set
$\range(\bar{\psi}(A^*A))$ becomes. Hence the condition (\ref{Grange}) is a
strong one for higher rates $\bar \psi$ and vice versa.

An alternative link condition is
\begin{equation} \label{Gineq}
\|\bar{\varrho}(G)x\| \le C\, \|Ax\| \quad \mbox{for all} \;\; x \in X
\end{equation}
for some constant $C>0$, where the rate function $\bar \varrho$ acts as a
benchmark for the {\sl degree of ill-posedness} of equation (\ref{opeq}) with
respect to the a priori information (\ref{Gsmooth}). From Proposition 2.1 in
\cite{Boetal06} we know that a range inclusion $\range(G_1) \subseteq
\range(G_2)$ and a condition of the form $\|G_1x\| \le C\|G_2x\|$ for all $x \in
X$ and some $C>0$ are equivalent. Consequently, with Proposition 2.18 in
\cite{EHN96} the condition (\ref{Gineq}) is equivalent to the range inclusion
\begin{equation} \label{rhorange}
\range(\bar \varrho(G)) \subseteq \range(A^*)=\range((A^*A)^{\scriptscriptstyle
1/2})
\end{equation}
taking into account the identity $\|Ax\|=\|(A^*A)^{1/2}x\|$ for all $x \in X$.

As exploited in \cite[\S~4]{Boetal06} one can reduce (\ref{rhorange}) to the
form (\ref{Grange}) with $\bar \psi(t)=\bar \varrho^{-1}(\sqrt{t})$ if $[\bar
\varrho^{-1}(\sqrt{t})]^2$ is an {\sl operator monotone} function
(cf.~\cite{Bhatia97}. An important special case of that implication, namely for
$\varrho(t)=t^{\frac{1}{2\mu}}$ with $0<\mu \le 1/2$, is well-known as
Heinz-Kato inequality (see, e.g., [6, Proposition 8.21] or the corollary of
Theorem~2.3.3 in \cite[p. 45]{Tanabe79}). In that special case, (\ref{rhorange})
yields (\ref{Grange}) with $\bar \psi(t)=t^\mu$ for exponents $0<\mu \le 1/2$.

In the next section we review the definition and some properties of
index functions and variable Hilbert scales. The fundamental interpolation
inequality is given with a short proof together with an application to
a general regularisation method. We then show how the variable Hilbert
scales provide natural source conditions. In the third section bounds
for the modulus of continuity are given in a variable Hilbert scale
setting. An important part of this section compares the new bounds
on the modulus of continuity with some obtained earlier and shows how
the new results have a substantially simpler structure. The fourth
section analyses linear regularisation methods and parameter choices
using the variable Hilbert scale approach. In section 5 we consider
several examples from image processing and spectral enhancement and
the paper finishes with some final remarks.

  \section{Interpolation inequalities and consequences}

  The main tool used here to derive error bounds for regularised solutions is an
  extension of interpolation inequalities to variable Hilbert scales. For
  classical Hilbert scales $\{X_r\}_{r \in \mathbb R}$ -- with real numbers as
  scale index $r$ -- interpolation inequalities are well-established. These
  interpolation inequalities were initially applied to the treatment of linear
  ill-posed problems (\ref{opeq}) by Natterer in~\cite{Nat84} (see also the
  monograph by Engl, Hanke and Neubauer \cite{EHN96}). For variable Hilbert
  scales, new interpolation inequalities have to be formulated. Here the scale
  index $r$ is replaced by a wide class of index functions defined as:

  \begin{definition} \label{def1}
  We call a real function $\theta$ defined on the open interval $(0,\infty)$ an
  \emph{index function} if it is continuous and positive. The \emph{index set}
  denoted by $\I$ is then the set of all such index functions.

  We call an index function $\bar \theta \in \I$ a \emph{rate function}
  if it is monotonically increasing and if it satisfies the limit condition
  $\lim \limits_{t \to +0} \bar \theta(t)=0.$
  \end{definition}

  {\parindent0em Note} that any monotonically increasing continuous function
  $\theta(t)$ defined on a finite interval $(0,t_0]$ satisfying $\lim
  \limits_{t \to +0} \theta(t)=0$ can be extended to a rate function $\bar
  \theta$ such that $\bar \theta(t) = \theta(t)$ for $t\in (0,t_0]$.
  Furthermore, the index functions corresponding to the classical Hilbert scales
  $X_r$ can be seen to be power functions $\theta(\lambda) = \lambda^r$ for real
  $r$. Rate functions are obtained for this case if $r>0$.

  The set of index functions $\I$ includes the positive constant functions and
  all power functions but not the zero function. We denote the pointwise
  operations by $\phi+\psi$, $\phi\psi$ and $\phi/\psi$, respectively. As usual,
  multiplication by a constant $\gamma$ is denoted by $\gamma \phi$. The
  composition is denoted by $\phi\circ \psi$ where $(\phi\circ\psi)(\lambda) =
  \phi(\psi(\lambda ))$. The pointwise maximum of two index functions is $\phi
  \vee \psi$ defined by $(\phi \vee \psi)\, (\lambda) = \max(\phi(\lambda),
  \psi(\lambda))$ and the pointwise minimum is $\phi \wedge \psi$. One verifies
  that the index set $\mathcal{I}$ from Definition~\ref{def1}
  is closed under
  \begin{itemize}
    \item point-wise addition, multiplication and division,
    \item multiplication with positive constants,
    \item pointwise maximum and minimum and
    \item composition.
  \end{itemize}
  If an index function is injective and surjective, the inverse denoted by
  $\phi^{-1}$ is also an index function. Not every index function is invertible,
  however. The reciprocal function of $\phi$ (with values $1/\phi(\lambda)$) is
  denoted by $1/\phi$.

  The variable Hilbert scales are then families of Hilbert spaces indexed by
  $\I$.
  \begin{definition} \label{def2}
    For a given injective self-adjoint positive definite linear operator $T$
    densely defined on a Hilbert space $X$ we define the \emph{variable Hilbert
    scale} $\{X_\theta\}_{\theta \in I}$ as a family of Hilbert spaces
    $X_\theta$ indexed by functions $\theta$ from the set $\I$ of index functions
    in the sense of Definition~\ref{def1}. Every Hilbert space $X_\theta$ with
    $\theta \in I$ is then the closure of the domain of the quadratic form
  \begin{equation}
    \label{theta-norm}
    \|f\|_\theta^2 = (f, \theta(T)\,f)\,.
  \end{equation}
  \end{definition}

  \medskip

  In such a way, variable Hilbert scales were introduced by Hegland
  in~\cite{Heg92} for the special case of $T$ being the inverse of a compact
  operator and in~\cite{Heg95} for more general $T$. The choice of the operator
  $T$ determines the Hilbert scale. In the simplest case where both $T$ and its
  inverse $T^{-1}$ are bounded all the Hilbert spaces $X_\theta$ are isomorphic
  to $X$ because of the continuity of the index functions $\theta$. In this
  paper, in the context of ill-posed problems (\ref{opeq}) we often assume that
  $T$ \emph{is unbounded but has a bounded inverse}, i.e.~the spectrum of $T$ is
  contained in the interval $[\|T^{-1}\|^{-1},\infty)$ and has $+\infty$ as an
  accumulation point. As the function $1/\lambda$ is an index function and the
  set of index functions is closed under composition, the inverse $T^{-1}$
  generates the same Hilbert scale as $T$. It is thus not necessary to consider
  variable Hilbert scales generated by invertible $T$ and bounded $T^{-1}$
  separately. The more general case where both $T$ and the inverse $T^{-1}$ are
  unbounded is only considered for the negative Laplacian $T=-Delta$ and in
  particular $T=-d^2/dt^2$. For the more general case where also $A$ is
  unbounded we refer to the recent paper \cite{HofMW09}. To get a link with
  (\ref{opeq}), a particular $T$ is suggested either by the forward operator
  $A$, by the operator $G$ of condition (\ref{Grange}) or based on a combination
  of both. A common choice is $T=(A^*A)^{-1}$ for injective operators $A$ with a
  non-closed range. It follows that $A^*A=\theta(T)$ if
  $\theta(\lambda)=1/\lambda$. For classes of problems connected with
  deconvolution, however, $T=-d^2/dx^2$ on $L_2(\R)$ is the canonical choice as
  $T$ is the generator of symmetric convolutions. An index function $\theta$
  such that $A^*A=\theta(T)$ is then found using Fourier transforms. More
  generally, for problems where the source conditions relate to smoothness,
  $T=-\Delta$ can be chosen. In such a case $\Delta$ denotes the Laplacian on
  $L_2(\Omega)$ for some domain $\Omega\subseteq \R^d$.

  It was shown in~\cite{Heg95} that there exists a continuous embedding $X_\phi
  \hookrightarrow X_\psi$ if and only if $\phi\leq\gamma\,\psi$ for some
  constant $\gamma>0$. If two different index functions $\theta_1$ and
  $\theta_2$ are identical on the spectrum of $T$ they define the same norms and
  hence the same space $X_{\theta_1}=X_{\theta_2}$. If they differ on the
  spectrum they do define different norms, however, these norms may be
  equivalent and thus the Hilbert spaces $X_{\theta_1}$ and $X_{\theta_1}$ as
  elements of the variable Hilbert scale $\{X_\theta\}_{\theta \in I}$ are
  indistinguishable.

  The most important connection between the norms of different spaces $X_\theta$
  is the \emph{interpolation inequality for variable Hilbert scales}.

  \begin{lemma}[Interpolation inequality] \label{lem1}
   Let $T$ be an unbounded injective self-adjoint positive definite linear
   operator densely defined on the Hilbert space $X$ with bounded inverse
   $T^{-1}: X \to X$. Moreover let $\phi,\psi,\theta$ and $\Psi$ be index
   functions such that $\Psi$ is concave and
    \begin{equation}
      \label{condition1}
      \phi(\lambda) \leq \Psi(\psi(\lambda)), \quad
      \text{for $\;\|T^{-1}\|^{-1} \le \lambda<\infty$}.
    \end{equation}
    Then for any element $0 \not= f \in X_\theta\cap X_{\psi\theta}$ one gets $f
    \in X_{\phi \theta}$ and
    \begin{equation}
      \label{interpolation}
      \frac{\|f\|_{\phi\theta}^2}{\|f\|_\theta^2} \leq
      \Psi\left(\frac{\|f\|_{\psi\theta}^2}{\|f\|_\theta^2}\right).
    \end{equation}
  \end{lemma}

  \begin{proof}
    Let in the following the measure $\nu$ be defined by
    $$d\nu(\lambda) = \|f\|_\theta^{-2} \theta(\lambda)\,
      d(f,E(\lambda)f) \quad \text{for $0 \not=f\in X_{\theta}$} $$
     where $E(\lambda)$ is the spectral family or resolution of the
    identity defined by $T$. By definition, because $\theta$ is positive and the
    integral of $d\nu$ equals 1, $\nu$ is a probability measure. Taking into
    account that $f \in X_{\psi\theta}$ and that $\Psi$ is concave we obtain
    from the inverse Jensen inequality that $\int_0^\infty \Psi(\psi(\lambda))
    \, d\nu(\lambda) \leq \Psi\left(\int_0^\infty \psi(\lambda)\,
    d\nu(\lambda)\right)< \infty$. Because integration is monotone and the
    inequality~\eqref{condition1} holds one gets $ \int_0^\infty \phi(\lambda)\,
    d\nu(\lambda)\leq \int_0^\infty \Psi(\psi(\lambda)) \, d\nu(\lambda)<\infty$
    and hence $f \in X_{\phi \theta}$. Summarising the results we arrive at the
    inequality $$ \int_0^\infty \phi(\lambda)\, d\nu(\lambda) \leq
    \Psi\left(\int_0^\infty \psi(\lambda)\, d\nu(\lambda)\right) $$ which
    provides us with the required inequality~\eqref{interpolation} by replacing
    $d\nu(\lambda)$ by its definition.
  \end{proof}

  The concavity of $\Psi$ is the key property which enables us to use Jensen's
  inequality. The Lemma~\ref{lem2} below shows that this property has only to be
  established for large arguments. We can focus on large arguments, if the
  spectrum of $T$ for $T$ under consideration contains only sufficiently large
  values and has $+\infty$ as an accumulation point. We need some auxiliary
  result:

  \begin{lemma}\label{lem:increase}
    If $\theta : [t_0,\infty) \rightarrow (0,\infty)$ is concave for some $t_0>
    0$ then $\theta$ is monotonically increasing. If moreover $\lim \limits_{t
    \to \infty} \theta(t)=\infty$, then $\theta$ is even strictly increasing.
  \end{lemma}
  \begin{proof}
    We show the contraposition. Assume that $\theta:[t_0,\infty)\rightarrow
    (0,\infty)$ is not monotonically increasing. Then there exist $t_0<t_1 <
    t_2$ such that $\theta(t_1)>\theta(t_2)$. Let
    $$
      l(t) = \frac{t-t_1}{t_2-t_1}\theta(t_2) + \frac{t_2-t}{t_2-t_1}\theta(t_1)
    $$
    be the linear interpolant of $\theta$ in $[t_1,t_2]$. As the slope of $l(t)$
    is $(\theta(t_2)-\theta(t_1))/(t_2-t_1) < 0$ one has $l(t)\rightarrow
    -\infty$ for $t\rightarrow \infty$. As $\theta(t) \geq 0$ there exists a
    $t_3 > t_2$ such that $\theta(t_3) > l(t_3)$. By rearranging this inequality
    one gets
    $$
      \theta(t_2) < \frac{t_2-t_1}{t_3-t_1}\theta(t_3) +
      \frac{t_3-t_2}{t_3-t_1}\theta(t_1)
    $$
    and so $\theta$ is not concave. The strict monotonicity for $\lim \limits_{t
    \to \infty} \theta(t)=\infty$ follows immediately from the fact that the
    hypograph of a concave function is a convex set.
  \end{proof}

  Now we can replace index functions which are concave for large arguments
  by such which are globally concave in the following way:

  \begin{lemma} \label{lem2}
    Let $\theta(\lambda)$ be an index function which is concave and hence by
    Lemma~\ref{lem:increase} increasing on the interval $0<\lambda_0 \le
    \lambda<\infty$. Then there exists an index function $\Psi(\lambda)$ which
    is concave for all $0<\lambda<\infty$ such that with some
    $\lambda_1>\lambda_0$ one has
    $$
      \Psi(\lambda) = \theta(\lambda), \quad \mbox{for}\;\; \lambda_1 \le
      \lambda < \infty\,,
    $$
    $$
      \Psi(\lambda) = \lambda \Psi(\lambda_1)/\lambda_1, \quad
      \mbox{for}\;\;0<\lambda \le \lambda_1\,.
    $$
     \end{lemma}
  \begin{proof}
  To obtain the assertion of this lemma we consider the set of real numbers
  $\{\alpha \mid \alpha \lambda \geq \theta(\lambda), \lambda_0 \le
  \lambda< \infty \}$. As $\theta(\lambda)$ is concave for $\lambda_0 \le
  \lambda< \infty$ this set is not empty and it is bounded below by zero. Thus
  it does have a greatest lower bound $\alpha_0 \geq 0$ such that
    \begin{itemize}
      \item $\alpha_0 \lambda \geq \theta(\lambda)$ for $\lambda_0 \le \lambda<
            \infty$,
      \item there is a $\lambda_1$ such that
            $\alpha_0\lambda_1=\theta(\lambda_1)$ if not, $\alpha_0$ would not
            be the greatest lower bound.
    \end{itemize}
    Hence, knowing from Lemma~\ref{lem:increase} that an index function $\Psi$
    which is concave for all $0<\lambda<\infty$ is always increasing, the
    function $\Psi(\lambda)$ can be composed of a linear function growing from
    zero to $\theta(\lambda_1)$ in the interval $(0,\lambda_1]$ and coinciding
    with $\theta$ for greater arguments.
  \end{proof}

  The interpolation inequality is the main tool to obtain error bounds for
  solvers of linear ill-posed problems. However, taking into account
  Lemma~\ref{lem2} by inspection it becomes clear that rate results derived from
  Lemma~\ref{lem1} are only based on the behaviour of $\Psi(\lambda)$ for large
  $\lambda \ge \lambda_1$.
  Without loss of generality $\Psi$ can be amended for $0<\lambda \le \lambda_1$
  by the linear function $\Psi(\lambda) = \Psi(\lambda_1)\lambda/\lambda_1$ for
  $0<\lambda < \lambda_1$.

   Three typical choices for $\Psi(\lambda)$ being concave at least for
   sufficiently large $\lambda$ are
  \begin{itemize}
    \item $\Psi(\lambda) = \lambda^\kappa$ where  $\kappa\in(0,1)$
    \item $\Psi(\lambda) = \lambda/\log(\lambda)$
    \item $\Psi(\lambda) = \log(\lambda)$.
  \end{itemize}
  For all three choices we have the limit condition
  \begin{equation} \label{limquo}
  \lim \limits _{\lambda \to \infty} \frac{\Psi(\lambda)}{\lambda} \,=\,0
  \end{equation}
 and one gets the following versions of
  interpolation inequalities from Lemma~\ref{lem1}:
  \begin{itemize}
    \item For $\Psi(\lambda)=\lambda^\kappa$ one gets
      $$
        \|f\|_{\phi\theta} \leq \|f\|_\theta^{1-\kappa}\,
        \|f\|_{\psi\theta}^{\kappa}\,,
      $$
    \item for $\Psi(\lambda) = \lambda/\log(\lambda)$ one gets
      $$
        \|f\|_{\phi\theta} \leq \frac{\|f\|_{\psi\theta}}%
        {\sqrt{2\log(\|f\|_{\psi\theta}/\|f\|_\theta)}}\,,
      $$
    \item and for $\Psi(\lambda) = \log(\lambda)$ one has
      $$
        \|f\|_{\phi\theta} \leq \|f\|_\theta
        \sqrt{2\log(\|f\|_{\psi\theta}/\|f\|_\theta}.
      $$
  \end{itemize}

  Asymptotically, i.e.~for $\|f\|_\theta\rightarrow 0$, the interpolation
  inequality allows us to find error bounds in the application to the error
  estimation for the solution of equation (\ref{opeq}). One aims to get bounds
  for the norm $\|f\|$ in $X$ using values of the image norm $\|Af\|$ in $Y$ and
  values of the norm $\|f\|_{\psi\theta}$ which expresses the specific
  additional smoothness of $f$. The terms in the interpolation
  inequality~\eqref{interpolation} are then
  $$
    \|f\|_{\phi\theta} = \|f\| \;\; \mbox{for} \;\; f\in X_{\phi\theta} \quad
    \mbox{and} \quad \|f\|_\theta = \|Af\| \;\; \mbox{for} \;\; f\in X_\theta.
  $$
  The first condition leads to $\phi(\lambda)\theta(\lambda)=1$ for all
  $\lambda$ and the second condition gives $\theta(T)= A^*A$ and with
  $\theta(\lambda):=1/\lambda$ the relations $T = (A^*A)^{-1}$ and
  $\phi(\lambda) = \lambda$. We are still free to choose the index functions
  $\psi$ and do it in the form $\psi(\lambda):= \chi(\lambda)\,\lambda$ with an
  appropriate index function $\chi$.

 For later use we add here some observations about convex functions which are
 stated as a lemma:
  \begin{lemma} \label{concave1}
    Let $\Psi :(0,\infty) \rightarrow (0,\infty)$ be a concave function. Then we
    have the following properties:
    \begin{itemize}
    \item[(a)] The function $\Xi :(0,\infty) \rightarrow (0,\infty)$ defined by
    \begin{equation} \label{concaveprop1}
      \Xi(\lambda) := \frac{\Psi(\lambda)}{\lambda},\qquad 0<\lambda<\infty
    \end{equation}
    is monotonically decreasing.
    \item[(b)] The function $\Phi :(0,\infty) \rightarrow (0,\infty)$ defined by
        \begin{equation} \label{concaveprop2}
      \Phi(\mu) :=  \mu \, \Psi \left(\frac{1}{\mu}\right),\qquad 0<\mu<\infty
    \end{equation}
    is concave and hence monotonically increasing.
    \end{itemize}
  \end{lemma}
  \begin{proof}
  (a)\, Let $0<\lambda_0<\lambda_1<\lambda_2$. As $\Psi$ is concave and positive
  one has
    \begin{align*}
      \Psi(\lambda_1) &
      \geq\frac{\lambda_1-\lambda_0}{\lambda_2-\lambda_0}\,\Psi(\lambda_2) +
      \frac{\lambda_2-\lambda_1}{\lambda_2-\lambda_0}\, \Psi(\lambda_0) \\
      & \geq\frac{\lambda_1-\lambda_0}{\lambda_2-\lambda_0}\,\Psi(\lambda_2).
    \end{align*}
    As this holds for arbitrarily small $\lambda_0>0$ on has
    $$
      \Psi(\lambda_1) \geq \frac{\lambda_1}{\lambda_2}\, \Psi(\lambda_2)
    $$
    and consequently $\Xi(\lambda_1) \geq \Xi(\lambda_2)$. This proves assertion
    (a) of the lemma.

    (b) \, Let $0 < \mu_0 < \mu_1 < \mu_2$ and $\lambda_i = 1/\mu_i$. Then one
    has $0 < \lambda_2 < \lambda_1 < \lambda_0$ and by the concavity of $\Psi$
    and some simple algebraic manipulations one gets
     \begin{align*}
       & \frac{\mu_1-\mu_0}{\mu_2-\mu_0} \Phi(\mu_2) +
       \frac{\mu_2-\mu_1}{\mu_2-\mu_0} \Phi(\mu_0) =
       \frac{\frac{1}{\lambda_1}-\frac{1}{\lambda_0}}%
       {\frac{1}{\lambda_2}-\frac{1}{\lambda_0}}
       \frac{\Psi(\lambda_2)}{\lambda_2} +
       \frac{\frac{1}{\lambda_2}-\frac{1}%
       {\lambda_1}}{\frac{1}{\lambda_2}-\frac{1}{\lambda_0}}
       \frac{ \Psi(\lambda_0)}{\lambda_0} \\
       & = \frac{1}{\lambda_1} \left(\frac{\lambda_0-\lambda_1}%
       {\lambda_0-\lambda_2}\Psi(\lambda_2) %
         + \frac{\lambda_1-\lambda_2}{\lambda_0-\lambda_2}
          \Psi(\lambda_0)\right) \\
       & \leq \frac{1}{\lambda_1}\, \Psi(\lambda_1) = \Phi(\mu_1).
     \end{align*}
     It follows that $\Phi$ is concave and hence by Lemma~\ref{lem:increase}
     also increasing. This completes the proof of the lemma.
  \end{proof}

\begin{remark} \label{remconnew}
   {\rm We note here that the transformation $\mathcal{S}: \,\Psi \in
   \mathcal{I} \mapsto \Phi \in \mathcal{I}$ according to formula
   (\ref{concaveprop2}),
   applicable to every index function and \emph{preserving concavity}, is an
   \emph{involution}, that means $\mathcal{S}^{-1}=\mathcal{S}$ and hence
   $\mathcal{S}$ is bijective.
   If the concave index function $\Psi$ satisfies $\lim\limits_{\lambda \to
   \infty} \Psi(\lambda)= \infty$, then by Lemma~\ref{lem:increase} the function
   is even strictly increasing and if, in addition, $\Psi$ is a rate function,
   i.e., it satisfies the additional limit condition $\lim\limits_{\lambda \to
   +0} \Psi(\lambda)=0$ (which is also motivated by Lemma~\ref{lem2}), the
   inverse function $\Psi^{-1}$ is a well-defined and convex index function.
   If, on the other hand, the limit condition (\ref{limquo}) holds, then we have
   $$
     \lim \limits_{\mu \to +0} \Phi(\mu)= \lim \limits_{\mu \to +0} \mu
     \,\Psi(1/\mu)= \lim \limits_{\lambda \to \infty} \Psi(\lambda)/\lambda=0
   $$
   and taking into account Lemma~\ref{concave1} (a) and (b) one sees that
   $\Phi=\mathcal{S}(\Psi)$ is a concave \emph{rate function}. Vice versa we
   have that $\Psi=\mathcal{S}(\Phi)$ satisfies (\ref{limquo}) whenever $\Phi$
   is a rate function.

   By inspection of the proof of Lemma~\ref{concave1} one can also see the
   following facts: If $\Psi(\lambda)$ is only concave for $\lambda \in
   [\lambda_0,\infty)$, then $\Phi(\mu)=[\mathcal{S}(\Psi)](\mu)$ is concave for
   $\mu \in (0,\mu]$ with $\mu_0=1/\lambda_0$. The involution $\mathcal{S}$
   preserves also the convexity of an index function and if the concavity or
   convexity is strict, then the strictness carries over to the transformed
   function.
   }\end{remark}

  Now we are ready to draw conclusions from Lemma~\ref{lem1}. A first, abstract
  version of bounds for errors of regularised solutions is given in the
  following corollary. We will denote by $f_\alpha$ an approximation of a
  solution $f$ to equation (\ref{opeq}) which is computed from an approximate
  right-hand side using a regularisation method and a regularisation parameter
  $\alpha>0$.
  \begin{corollary} \label{cor1}
  Let $A:X\rightarrow Y$ be an injective bounded linear operator with non-closed
  range mapping between the two Hilbert spaces $X$ and $Y$. Furthermore let the
  variable Hilbert scale $\{X_\nu\}_{\nu \in I}$ be generated by $T=(A^*A)^{-1}$
  such that any scale element $X_\nu$ has a norm denoted by $\|\cdot\|_\nu$.
  Moreover let $\chi$ and $\Psi$ be index functions and $\Psi$ be concave such
  that
  \begin{equation} \label{Psiineq}
     \Psi\left(\chi(\lambda)\,\lambda\right) \geq \lambda \quad
     \text{for all $\;\|T^{-1}\|^{-1} \le \lambda< \infty$}\,.
 \end{equation}
  If the solution $f$ to (\ref{opeq}) in addition satisfies the condition $f\in
  X_\chi$ and if $f_\alpha \in X_\chi$ is such that
  \begin{align}
    \|f_\alpha -f \|_\chi &= \zeta>0  \label{stability} \\
    \|Af_\alpha - g \| &= \epsilon>0 \label{consistency}
  \end{align}
  then
  \begin{equation} \label{upeps}
    \|f-f_\alpha\| \leq \epsilon \sqrt{\Psi(\zeta^2/\epsilon^2)}.
  \end{equation}
  \end{corollary}
  \begin{proof}
    By Lemma~\ref{lem1} with $\theta(\lambda)=1/\lambda$,
    $\phi(\lambda)=\lambda$, $\psi(\lambda)=\chi(\lambda)\,\lambda$
    and for a concave index function $\Psi$ satisfying (\ref{Psiineq}) one has
    for all $0 \not=h \in X_\chi$
    \begin{equation} \label{auxquo}
      \|h\|^2 \leq \|Ah\|^2\,
      \Psi\left(\frac{\|h\|_\chi^2}{\|Ah\|^2}\right).
    \end{equation}
 Setting $h:=f_\alpha-f$ this yields the estimate
$$ \|f_\alpha-f\|^2 \leq \|A(f_\alpha-f)\|^2 \,
\Psi\left(\frac{\|f_\alpha-f\|_\chi^2}{\|Af_\alpha-Af\|^2}\right)=\epsilon^2
\Psi((\zeta/\epsilon)^2)$$
    and proves the assertion of the corollary.
      \end{proof}

Results similar to those of Corollary~\ref{cor1} can be found for other choices
of $T$, see for example Corollary~\ref{cor1a} where $T=-d^2/dt^2$. The
Corollary~\ref{cor1} can be interpreted as an instance of an \emph{abstract Lax
theorem}~\cite{KenH09} where the condition~\eqref{stability} is a stability
condition and the bound~\eqref{consistency} relates to consistency.

\medskip

   Note that the error estimate (\ref{upeps}) of Corollary~\ref{cor1} requires
   the essential conditions $f \in X_\chi$ and $f_\alpha \in X_\chi$, i.e.~the
   approximate solutions $f_\alpha$ are constructed such that they obtain the
   same smoothness level with respect to $T$ as the exact solution $f$. A next
   step for drawing conclusions of Lemma~\ref{lem1} will be formulated in
   Corollary~\ref{cor11} by assuming that $f$ belongs to ball
   \begin{equation} \label{eq:ballchi}
   B_\chi(R):=\{h \in X_\chi:\; \|h\|_\chi \le R \}
   \end{equation}
   in $X_\chi$ with positive radius $R=R_1$ and that the approximate solutions
   $f_\alpha$ for all $\alpha>0$ under consideration belong to another such ball
   with radius $R=R_2$. Moreover, we consider for data $g^\delta$ satisfying
   (\ref{noise}) the limit process $\delta \to +0$ in correspondence with
   associated regularized solutions $f_\alpha$, where the regularisation
   parameter $\alpha>0$ is chosen either a priori as $\alpha=\alpha(\delta)$ or
   a posteriori as $\alpha=\alpha(\delta,g^\delta)$.

  \begin{corollary} \label{cor11}
  Under the setting of Corollary~\ref{cor1} let the limit condition
  (\ref{limquo}) be satisfied and let $f \in B_\chi(R_1),\;R_1>0$. Moreover with
  prescribed $\delta_{max}>0$ let $f_\alpha \in B_\chi(R_2),$ 
  $R_2>0,$ for all $\alpha$ attributed to $\delta \in (0,\delta_{max}]$ and
  $g^\delta$ satisfying (\ref{noise}) such that
  \begin{equation} \label{eq:ratexi}
  \|Af_\alpha-g\| \le \bar C\,\bar \xi(\delta), \qquad 0<\delta \le \delta_{max},
  \end{equation}
  for some rate function $\bar \xi$ and some constant $\bar C>0$.
  Then we have
   \begin{equation} \label{eq:specialxi}
   \|f-f_\alpha\|\, \le \bar C \bar
   \xi(\delta)\,\sqrt{\Psi\left(\left[\frac{R_1+R_2}{\bar C \bar
   \xi(\delta)}\right]^2\right)},\qquad 0<\delta \le \delta_{max}\,,
   \end{equation}
   where the upper bound in (\ref{eq:specialxi}) is a rate function, i.e., it
   tends to zero as $\delta \to 0.$
    \end{corollary}
  \begin{proof}
  Since $\Psi$ is concave by Lemma~\ref{lem:increase} the error norm  $\epsilon\,
   \sqrt{\Psi(\zeta^2/\epsilon^2)}$ obtained from (\ref{upeps}) is
  increasing in $\zeta>0$ for fixed $\epsilon>0$ and
  as a consequence of Lemma~\ref{concave1} (a) this upper bound   $\epsilon\,
   \sqrt{\Psi(\zeta^2/\epsilon^2)}= \zeta\,
   \sqrt{\frac{\Psi(\zeta^2/\epsilon^2)}{(\zeta/\epsilon)^2}}$ is increasing in
   $\epsilon>0$ for fixed $\zeta>0$.
   Moreover, due to (\ref{limquo}) in the limit process $\epsilon \to 0$ for
   fixed $\zeta>0$ implying $\zeta/\epsilon \to \infty$ this bound and hence the
   error norm in $X$ even tends to zero. For the mentioned kinds of monotonicity
   we obtain formula (\ref{eq:specialxi}) by $\epsilon \le \bar C \xi(\delta)$
   and $\zeta =\|f-f_\alpha\|_\chi\le \|f\|_\chi+ \|f_\alpha\|_\chi \le
   R_1+R_2$. The upper bound in (\ref{eq:specialxi}) is a rate function
   declining to zero as $\delta \to 0$ because $\bar \xi$ is a rate function.
      \end{proof}

   \begin{remark} \label{Remark_1}
   {\rm As a special case for the situation of Corollary~\ref{cor11} we can
   consider an a posteriori choice $\alpha_{dis}=\alpha_{dis}(\delta,g^\delta)$
   for the regularisation parameter realised by a 
   \emph{discrepancy principle}
   \begin{equation} \label{Morozov}
      \|Af_{\alpha_{dis}}-g^\delta\|= C_{dis} \delta
   \end{equation}
      with some prescribed $ C_{dis}>0$.
   Then by using the triangle inequality we obtain with~(\ref{noise}) as noise
  model $$ \|Af_{\alpha_{dis}}-g\| \le
  \|Af_{\alpha_{dis}}-g^\delta\|+\|g^\delta-g\| \le (C_{dis}+1)\,\delta=\bar C\,
  \delta\,. $$ Then for such $\alpha=\alpha_{dis}$ under (\ref{limquo}) the
  regularisation method converges strongly in $X$ with the convergence rate
  \begin{equation} \label{specialrate1}
  \|f-f_\alpha\|\,=\,\mathcal{O}\left(\delta\,\sqrt{\Psi(\bar
  K/\delta^2)}\right) \quad \mbox{as} \quad \delta \to 0 \end{equation} for some
  constant $\bar K>0$. Note that beside the assumption $f \in B_\chi(R_1)$ on
  the solution smoothness for that result the strong condition
  $f_{\alpha(\delta,g^\delta)}\in B_\chi(R_2)$ for all $\delta \in
  (0,\delta_{max}]$ and all associated $g^\delta$ satisfying (\ref{noise}) is
  required.
   }\end{remark}

 The convergence rate in (\ref{specialrate1}) depends only on the asymptotic
 behaviour of $\Psi(\lambda)$ as $ \lambda \to \infty$. Thus the alteration of
 $\Psi(\lambda)$ for small $\lambda$ in the sense of Lemma~\ref{lem2} has no
 influence on that rate. For the class of functions
 $\Psi(\lambda)=\lambda^\kappa$ with $0<\kappa<1$ rate functions proportional to
 $\delta^{1-\kappa}$ occur in (\ref{specialrate1}). All those error rates are
 lower than the rate $\delta$ which is typical for well-posed problems. It
 should be mentioned that $\Psi(\lambda)=\lambda$ fails to satisfy the condition
 (\ref{limquo}) and used in Corollary~\ref{cor1} the inequality (\ref{auxquo})
 does not yield a convergence rate.

  To get a feeling for the role of the solution smoothness $f \in X_\chi$ we can
  study consequences of the inequality (\ref{Psiineq}) as a hypothesis of
  Corollary~\ref{cor1} taking into account Lemma~\ref{lem:increase}. One
  consequence of (\ref{Psiineq}) is the limit condition $\lim_{\lambda \to
  \infty} \Psi(\lambda)=\infty$ for the function $\Psi$ which is because of its
  concavity then strictly increasing and invertible with convex
  $\Psi^{-1}(\lambda)$ also tending to infinity as $\lambda \to \infty$. Then
  (\ref{Psiineq}) implies $\chi(\lambda) \ge \frac{\Psi^{-1}(\lambda)}{\lambda}$
  for large $\lambda$. Under that condition (\ref{limquo}) is equivalent to
  $\lim_{\lambda \to \infty} \Psi^{-1}(\lambda)/\lambda= \infty$. Hence, the
  index function $\chi(\lambda)$ tends to infinity for $\lambda \to \infty$
  provided that (\ref{limquo}) holds true.

  When setting $\phi(\lambda):=\lambda,\;\theta(\lambda):=1/\lambda$ and
  $\psi(\lambda):=\Psi^{-1}(\lambda)$ in the interpolation inequality
  (\ref{interpolation}) then the corresponding \emph{regularity condition} $f
  \in X_{\psi\theta}$ is equivalent to a \emph{source condition}
  (\ref{scclassic}) which expresses the specific smoothness of the solution $f$
  with respect to the forward operator $A$ of equation (\ref{opeq}).
  \begin{proposition}
  \label{prop1}
  Let $\Psi(\lambda)$, for $0<\lambda<\infty$, be a concave and strictly
  increasing index function satisfying the limit conditions $\lim
  \limits_{\lambda \to +0} \Psi(\lambda)=0,\;$ $\lim \limits_{\lambda \to
  \infty} \Psi(\lambda)=\infty$ and (\ref{limquo}). Moreover let $T=(A^*A)^{-1}$
  and set $\phi(\lambda):=\lambda,\;\theta(\lambda):=1/\lambda$ as well as
  $\psi(\lambda):=\Psi^{-1}(\lambda)$ for $0<\lambda<\infty$. Then we have $f
  \in X_{\psi\theta}$ if and only if $f$ satisfies a source condition
  (\ref{scclassic}) with the function
  \begin{equation} \label{tildeop}
    \bar{\psi}(t) = \frac{1}{\sqrt{t\Psi^{-1}(1/t)}},\qquad 0<t<\infty,
  \end{equation}
  which is then a rate function.
    \end{proposition}
  \begin{proof}
  Under the stated assumptions the function $\bar{\psi}$ is well-defined and a
  rate function. Namely, we can write
  $\frac{1}{\sqrt{t\Psi^{-1}(1/t)}}=\sqrt{\frac{\Psi(u)}{u}}$ when using the
  substitution $u:=\Psi^{-1}(1/t)$. The variable $u>0$ is strictly decreasing
  with respect to $t>0$ such that $u \to \infty$ corresponds with $t \to +0$ and
  vice versa $t \to \infty$ corresponds with $u \to +0$, because $\Psi^{-1}$ is
  also strictly increasing and we have $\lim \limits_{\lambda \to \infty}
  \Psi^{-1}(\lambda)=\infty$ and $\lim \limits_{\lambda \to +0}
  \Psi^{-1}(\lambda)=0$ for the functions $\Psi$ under consideration. Now by
  (\ref{limquo}) we have $\lim \limits_{u \to \infty}\frac{\Psi(u)}{u}=0$ and
  with Lemma~\ref{concave1} (a) the quotient $\frac{\Psi(u)}{u}$ is
  monotonically decreasing in $u>0$. This, however, implies that $\bar \psi(t)$
  is monotonically increasing for $t>0$ with limit condition $\lim \limits_{t
  \to +0} \bar \psi(t)=0$. Hence, $\bar \psi$ is a rate function.

   Moreover, we have
  $$f\in X_{\psi\theta} \quad \Longleftrightarrow \quad
  (f,\Psi^{-1}(T)T^{-1}f)<\infty$$ and $$
    f=\bar{\psi}(A^*A)v, \;\; \mbox{for} \;\; v\in X \quad \Longleftrightarrow
    \quad ([\bar{\psi}(A^*A)]^{-1}f,[\bar{\psi}(A^*A)]^{-1}f) <\infty\,.
 $$
  One has
  equivalence if and only if
  $$
    [\Psi^{-1}((A^*A)^{-1})](A^*A) = [\bar{\psi}(A^*A)]^{-2}
  $$
  and the claim follows. This proves the proposition.
  \end{proof}

\bigskip

    After the millennium {\sc Math\'e and Pereverzev} with coauthors seized,
  reused and extended {\sc Hegland}'s ideas and concepts of variable Hilbert
  scales and corresponding interpolation inequalities from \cite{Heg92,Heg95}
  for linear ill-posed problems and their regularisation
  (cf.~\cite{MatPer03b,MatPer03a}) and combined it
  (cf.~\cite{HofMAt07,MatTau06}) with the concept of {\sl approximate source
  conditions} (cf.~\cite{DHY07,Hof06}). The comprehensive theory developed
  therein considers only rate functions as index functions. Such an approach
  leads in general to different formulae compared with the results based on the
  concept of Lemma~\ref{lem1}, but as we will outline in the sequel clear
  cross-connections and sometimes even equivalences of the assertions obtained
  characterise the two different ways.

  \section{Modulus of continuity of $A^{-1}$}

  The \emph{modulus of continuity} of $A^{-1}$ restricted to the set $AM$ with
  $M \subseteq X$ is
  $$
    \omega(M,\delta) = \sup \{\|x\|:\; x\in M, \; \|Ax\| \leq \delta\}.
  $$
  The impact of the modulus of continuity on error bounds in regularisation has
  recently been discussed in the paper \cite[\S 4]{HofMS08}. It is well-known
  that the worst case error
  $$
    e(\hat f,M,\delta):= \sup \limits _{f \in M} e(\hat f,f,\delta) \qquad
    \mbox{for} \qquad e(\hat f,f,\delta):= \sup \limits _{g^\delta \in
    Y:\,\|Af-g^\delta\|\le \delta} \|\hat f(g^\delta) -f\|
  $$
  of linear and nonlinear reconstruction methods $\hat f:\, g^\delta \in Y
  \mapsto \hat f (g^\delta) \in X$ has an infimum
  $$
    e(M,\delta):= \inf \limits _{\hat f:\;Y \to X} e(\hat f,M,\delta)
  $$
  which satisfies for centrally symmetric and convex sets $M$ the inequalities
  $$
    \omega(M,\delta) \le e(M,\delta) \le \omega(M,2\delta) \le 2
    \omega(M,\delta)\,.
  $$
  Hence the modulus of continuity $\omega(M,\delta)$ serves as benchmark for the
  reconstruction error of $\hat f$ when $f,\hat f \in M$ can be assumed. For
  example, from \cite[Lemma~4.2]{HofMS08} one can find a minimax-expression for
  the modulus of continuity in the case of centrally symmetric and convex {\sl
  source sets} for $M$ of the form
  \begin{equation} \label{Gset}
    M=G [B(R)]:=\{x \in X:\; x=Gv, \;v \in X,\; \|v\| \le R\}\,.
  \end{equation}
  corresponding to condition (\ref{Gsmooth}). This expression gets an explicit
  bound from above for the special case $G=\bar \psi(A^*A)$ and
  $$
    M=\bar \psi(A^*A)[B(R)]:=\{x \in X:\; x= \bar \psi(A^*A)v, \;v \in X,\;
    \|v\| \le R\}
  $$
  associated with the source condition (\ref{scclassic}). Note that the rate
  function $\bar \psi(t)$ is only of interest
  here for arguments $0<t \le \|A\|^2$, but without loss of generality
  (cf.~\cite[Theorem 1 (b)]{HofMW09}) we can extend $\bar \psi$ to be a
  monotonically increasing index function defined on $(0,\infty)$. Then by using
  the strictly increasing auxiliary function
  \begin{equation} \label{aux1}
    \Theta(t):= \sqrt{t}\, \bar \psi(t)\,, \qquad 0<t<\infty
  \end{equation}
  satisfying the limits conditions $\lim \limits_{t \to +0} \Theta(t)=0$ and
  $\lim \limits_{t \to \infty} \Theta(t)=\infty$ one obtains for
  $M=\bar\psi(A^*A)[B(R)]$
  \begin{equation} \label{MPoptirate}
    \omega(M,\delta)\,\le\,R\,\bar \psi
    \left(\Theta^{-1}\left(\frac{\delta}{R}\right) \right), \qquad \delta>0
  \end{equation}
   provided that
  \begin{equation} \label{concavecond1}
    \bar\psi^2((\Theta^2)^{-1}(t)) \quad \mbox{is concave for} \quad 0<t<\infty\,.
  \end{equation}
  This result can be derived from Corollary~3.7 and Theorem~2.1(c) in
  \cite{HofMS08} (see also Theorem~1 in the earlier paper \cite{MatPer03a}).
  A similar assertion was already mentioned in a rudimentary form in a paper by
  Ivanov and Korolyuk in 1969~\cite{IvaK69}.

  The following proposition also yields an upper bound for the modulus of
  continuity based on a variable Hilbert scale interpolation inequality using
  Lemma~\ref{lem1} or Corollary~\ref{cor1}. For the proof we use
  Lemma~\ref{concave1} (b).

  \begin{proposition} \label{prop_simpler}
    Let $\Psi(\lambda)$, for $0<\lambda<\infty$, be a concave and strictly
    increasing index function satisfying the limit conditions $\lim
    \limits_{\lambda \to +0} \Psi(\lambda)=0,\;$ $\lim \limits_{\lambda \to
    \infty} \Psi(\lambda)=\infty$ and (\ref{limquo}), for which an index
    function $\chi$ exists that satisfies
    $$
      \chi(\lambda) \ge \Psi^{-1}(\lambda)/\lambda, \qquad 0<\lambda <\infty.
    $$
    Furthermore let $X_\chi$ be an element of a Hilbert scale generated by
    $T=(A^*A)^{-1}$ where $A$ is injective. Then
  \begin{equation} \label{omega_ball}
     \omega(M,\delta) \leq \delta\,
     \sqrt{\Psi\left(\frac{R^2}{\delta^2}\right)},\qquad \delta>0
   \end{equation}
   for $M=B_{\chi}(R)$.
  \end{proposition}
  \begin{proof}
    Under the assumptions stated on $\Psi$ and $\chi$ Corollary~\ref{cor1}
    applies. Then from formula (\ref{auxquo}) we can conclude that $\|h\| \le
    \|Ah\| \sqrt{\Psi\left(\frac{\|h\|^2_\chi}{\|Ah\|^2}\right)}$ for all $0
    \not= h \in X_\chi$. As $\Psi$ is monotonically increasing one then gets
    $\|h\| \le \|Ah\| \sqrt{\Psi\left(\frac{R^2}{\|Ah\|^2}\right)}$. Now by
    Lemma~\ref{concave1} (a) the function $\Xi(\zeta)=\Psi(\zeta)/\zeta$ is
    monotonically increasing and so $\Xi(\zeta_1) \ge \Xi(\zeta_2)$ for
    $0<\zeta_1 \le \zeta_2<\infty$. This gives with
    $\zeta_1:=\frac{R^2}{\delta^2}$ and $\zeta_2:=\frac{R^2}{\|Ah\|^2}$ the
    estimate $\|h\| \le \delta\, \sqrt{\Psi\left(\frac{R^2}{\delta^2}\right)}$
    for all $h \in B_{\chi}(R)$ satisfying the additional condition $\|Ah\| \le
    \delta$. Thus the proposition is proven.
  \end{proof}

  We note that for centrally symmetric and convex sets $M, \;f \in M,$ and
  regularised solutions $f_{\alpha_{\text{dis}}} \in M$ obtained from the
  discrepancy principle of form (\ref{Morozov}) mentioned in
  Remark~\ref{Remark_1} we easily derive along the lines
  of~\cite[Lemma~2.2]{HofMS08} that
  \begin{equation} \label{quasi1}
    \|f-f_{\alpha_{\text{dis}}}\| \le \omega(2M,(C_{\text{dis}}+1)\delta)
  \end{equation}
  with $2M:=\{u \in X:\,u=2v,\;v \in M\}$. In the case $M=B_{\chi}(R)$ with
  $2M=B_\chi(2R)$ the estimate (\ref{quasi1}) yields with (\ref{omega_ball}) a
  convergence rate of the form (\ref{specialrate1}) with constant $\bar
  K=4R^2/(C_{\text{dis}}+1)^2$. With more generality such rates were verified
  above directly from Corollary~\ref{cor1}.

\smallskip

   Under weak additional assumptions (see~\cite[Corollary~3.7]{HofMS08}) there
   is also a constant $\underline{C}>0$ such that $$
   \omega(B_{\chi}(R),\delta)\, \geq \, \underline C\,\delta\,
   \sqrt{\Psi\left(\frac{R^2}{\delta^2}\right)},\qquad \delta>0 \,.$$ Then a
   convergence rate of the form (\ref{specialrate1}) is \emph{order optimal}
   independent of the constant $\bar K>0$ because of
   $\sqrt{\Psi\left(\frac{C\,R^2}{\delta^2}\right)} \le \max\{C,1\}
   \sqrt{\Psi\left(\frac{R^2}{\delta^2}\right)}$ for all $C>0$. On the other
   hand, Corollary~\ref{cor11} yields an error estimate of best order just for
   $\xi(\delta) \sim \delta$, hence the discrepancy principle is order optimal
   in that sense.

\smallskip

  Evidently, under the assumptions of Proposition~\ref{prop1} with the
  additional setting
  \begin{equation} \label{connected}
   \chi(\lambda)\,:=\,\frac{\Psi^{-1}(\lambda)}{\lambda}\,=\,
   \frac{1}{\bar{\psi}(1/\lambda)^2}, \qquad 0<\lambda<\infty
  \end{equation}
  one has
  $$
     \bar\psi(A^*A)[B(R)] = B_\chi(R)
  $$
  where $B_\chi(R)$ denotes the ball (\ref{eq:ballchi}) of radius $R$ in
  $X_\chi$, an element of the Hilbert scale generated by $T=(A^*A)^{-1}$
  expressed through the index function $\chi$. We emphasise that the upper bound
  in (\ref{omega_ball}) for the modulus of continuity from
  Proposition~\ref{prop_simpler} needing only one function $\Psi$ has a
  \emph{much simpler structure} than the nested upper bound in
  (\ref{MPoptirate}) composing the functions $\bar \psi$ and $\Theta^{-1}$. Also
  the required concavity of $\Psi$ for obtaining (\ref{omega_ball}) \emph{looks
  much simpler} than the needed concavity of the composite function
   $$
     \bar\psi^2((\Theta^2)^{-1}(t))\,\equiv\,\bar\psi^2(\Theta^{-1}(\sqrt{t})),
     \qquad 0<t<\infty,
   $$
  for obtaining (\ref{MPoptirate}).

  Owing to the correspondence (\ref{tildeop}) between the concave index function
  $\Psi$ and the rate function $\bar \psi$ it is of some interest to compare the
  quality of the estimates (\ref{MPoptirate}) and (\ref{omega_ball}) as well as
  the strength of conditions which have to imposed in order to ensure those
  bounds for $\omega$.

 \begin{proposition} \label{prop:onetoone}
   Let $\Psi(\lambda)$, for $0<\lambda<\infty$, be a concave and strictly
   increasing index function satisfying the limit conditions $\lim
   \limits_{\lambda \to +0} \Psi(\lambda)=0,\;$ $\lim \limits_{\lambda \to
   \infty} \Psi(\lambda)=\infty$ and (\ref{limquo}). Then for the rate function
   $\bar \psi(t):=1/\sqrt{t\Psi^{-1}(1/t)}$ (cf.~(\ref{tildeop})) and by setting
   $\Theta(t):=\sqrt{t} \bar \psi(t),\;0<t<\infty$, we have the following
   assertions: The error bounds in (\ref{omega_ball}) and in (\ref{MPoptirate})
   and the corresponding concavity conditions required for obtaining those
   bounds coincide, i.e., we have
   \begin{equation} \label{coin}
      \delta\, \sqrt{\Psi\left(\frac{R^2}{\delta^2}\right)} \, =\, R\,\bar \psi
      \left(\Theta^{-1}\left(\frac{\delta}{R}\right) \right)\,, \qquad
      R>0,\quad\delta>0\,.
   \end{equation}
   Moreover, the function $\bar\psi^2((\Theta^2)^{-1}(t))$ is concave for all
   $0<t<\infty$.

   Vice versa, any rate function $\bar \psi(t),\;0<t<\infty,$ determines by
   equation (\ref{tildeop}) in a unique manner a strictly increasing index
   function $\Psi(\lambda),\,0<\lambda<\infty,$ satisfying the limit conditions
   $\lim \limits_{\lambda \to +0} \Psi(\lambda)=0,\;$ $\lim \limits_{\lambda \to
   \infty} \Psi(\lambda)=\infty$ and (\ref{limquo}) which is concave for all
   $0<\lambda<\infty$ if $\bar\psi^2((\Theta^2)^{-1}(t))$ is concave for all
   $0<t<\infty$ which again implies the coincidence (\ref{coin}) of the error
   bounds.
 \end{proposition}
\begin{proof}
   First we find from Proposition~\ref{prop1} that $\bar \psi(t),\;t>0,$ is a
   rate function if $\Psi(\lambda),\,0<\lambda<\infty$ is a concave and strictly
   increasing index function satisfying the limit conditions $\lim
   \limits_{\lambda \to +0} \Psi(\lambda)=0,\;$ $\lim \limits_{\lambda \to
   \infty} \Psi(\lambda)=\infty$ and (\ref{limquo}). Then from the right
   equation in (\ref{connected}) (cf.~(\ref{tildeop})) we have
   $\Psi^{-1}(\lambda)=\lambda/\bar \psi^2(1/\lambda)$. By using the bijective
   substitution $u=1/\lambda$ in $(0,\infty)$ this yields
   $\Psi^{-1}(1/u)=\frac{1}{\Theta^2(u)}$ and $\frac{1}{u}=
   \Psi\left(\frac{1}{\Theta^2(u)}\right)$ for $0<u<\infty$. Multiplying the
   last equation by the factor $u\, \bar \psi ^2(u)$ we derive $${\bar
   \psi}^2(u)=u \,{\bar \psi}^2(u) \Psi\left(\frac{1}{\Theta^2(u)}\right)=
   \Theta^2(u) \Psi\left(\frac{1}{\Theta^2(u)}\right)$$ and $\bar \psi(u)=
   \Theta(u) \sqrt{\Psi\left(\frac{1}{\Theta^2(u)}\right)}$. By exploiting the
   bijection $t=\Theta(u)$ of $(0,\infty)$ into itself this provides us with the
   equation $ \bar \psi \left(\Theta^{-1}(t) \right)= t \sqrt{\Psi
   \left(\frac{1}{t^2} \right)}$ which implies the required identity
   (\ref{coin}) by inserting $t:=\delta/R$ and multiplying the arising equation
   by $R$.

    In a second step we note that by using the monotonically increasing
    bijection $s=\Theta^2(u)$ between $s \in (0,\infty)$ and $u \in (0,\infty)$
    and once more by exploiting the right equation in (\ref{connected}) we can
    write as follows for all $s>0$:
    $$\bar \psi^2((\Theta^2)^{-1}(s)) = \bar \psi^2(u) =\frac{\Theta^2(u)}{u}=
      \Theta^2(u)\,\Psi \left( \frac{1}{\Theta^2(u)}\right) = s
      \,\Psi\left(\frac{1}{s}\right)
    =[\mathcal{S}(\Psi)](s).
    $$
     Hence, by Lemma~\ref{concave1} (b) we immediately see that as required
    $\bar \psi^2((\Theta^2)^{-1}(s)),\,s>0,$ is concave if
    $\Psi(\lambda),\,\lambda>0,$ is concave.

     Since the involution $\mathcal{S}$ (cf.~Remark~\ref{remconnew}) preserves
     concavity, the reverse assertion formulated in
     Proposition~\ref{prop:onetoone} becomes immediately clear, since
     (\ref{tildeop}) represents a one-to-one correspondence between index
     functions $\bar \psi$ and strictly increasing functions $\Psi$ with the
     limit conditions under consideration.
    \end{proof}

We now investigate the concavity condition for the function $\bar
\psi^2((\Theta^2)^{-1}(s))$ in more detail. For this a characterisation of the
concavity of index functions is given in terms of the monotonicity of certain
divided differences.

\begin{lemma}
  \label{diffquot}
  Let $\psi$ be an index function. Then the three following statements are
  equivalent:
  \begin{enumerate}
    \item $\psi$ is concave
    \item $(\psi(s_0+s)-\psi(s_0))/s$ is a decreasing index function for all
    $s_0>0$
    \item $(\psi(s_0)-\psi(s_0-s))/s$ is an increasing continuous function
    $(0,s_0)\rightarrow \R_+$ for all $s_0>0$.
  \end{enumerate}
\end{lemma}
\begin{proof}
  If $\psi$ is a concave index function then by Lemma~\ref{lem:increase} $\psi$
  is increasing and so both $(\psi(s_0+s)-\psi(s_0))/s$ and
  $(\psi(s_0)-\psi(s_0-s)/s$ are positive continuous functions for $s>0$ and
  $s\in (0,s_0)$, respectively. Furthermore by definition
  $$
    (t_2-t_0) \psi(t_1) \geq (t_2-t_1) \psi(t_0) + (t_1-t_0) \psi(t_2)
  $$
  and by simple algebraic manipulations and the right choice of $t_0< t_1 < t_2$
  one gets the second and third statement from the first.

  Conversely, if $(\psi(s_0+s)-\psi(s_0))/s$ is a decreasing index function for
  all $s_0>0$ one has for all $t_0< t_1 < t_2$
  $$
    \frac{\psi(t_1)-\psi(t_0)}{t_1-t_0} \geq \frac{\psi(t_2)-\psi(t_0)}{t_2-t_0}
  $$
  and thus $\psi$ is concave. A similar argument shows that $\psi$ is concave
  if the third statement holds.
\end{proof}
A direct consequence of this lemma is that for concave rate functions $\bar\psi$
one has
$$  \frac{\bar\psi(s_0)-\bar\psi(s_0-s)}{s} \leq \frac{\bar\psi(s_0)}{s_0} $$
as $\lim_{s\rightarrow 0}\, \bar\psi(s)=0$. Another consequence is
\begin{proposition}
  If $\psi(t)$ is a concave rate function then so is $\psi(\sqrt{t})^2$.
\end{proposition}
\begin{proof}
   By lemma~\ref{diffquot} we have to show that for all $t_0>0$ the function
   $(\psi(\sqrt{t+t_0})^2-\psi(t-0)^2)/t$ is a decreasing index function. As the
   mapping $s\rightarrow (s+s_0)^2$ is monotone it is sufficient to show that
   $$
     \omega(s) = \frac{\psi(s+s_0)^2-\psi(s_0)^2}{(s+s_0)^2 - s_0^2}
   $$
   is monotonically decreasing.

   As $\psi$ is assumed to be concave, Lemma~\ref{diffquot} implies that
   $$
    \sigma(s) = \frac{\psi(s+s_0)-\psi(s_0)}{s}
   $$
   is monotonically decreasing. Furthermore
   $$
     \omega(s) = \sigma(s) \left(\frac{\psi(s+s_0)+\psi(s_0)}{s+2s_0}\right) =
       \sigma(s) \frac{s\sigma(s)+ 2\psi(s_0)}{s+2s_0}.
   $$
   Now let $s_2< s_2$. As $\sigma(s)$ is monotonically decreasing on has
   $$
     \omega(s_1) \geq \sigma(s_2) \frac{s_1\sigma(s_2)+2\psi(s_0)}{s_1+2s_0}
       = \sigma(s_2)^2 \frac{s_1+2\psi(s_0)/\sigma(s_2)}{s-1+2s_0}.
   $$
   The right-hand side is a decreasing function of $s_1$ if
   $2 s_0 \leq 2 \psi(s_0)/\sigma(s_2)$, i.e., $\sigma(s-2) \leq \psi(s_0)$.
   This is a consequence of Lemma~\ref{diffquot} as stated in the remark after
   the lemma.
   Replacing $s_1$ by $s_2$ thus gives a lower bound for $\omega(s_1)$ and
   thus
   $$
     \omega(s_1) \geq \sigma(s_2)^2 \frac{s_2+2\psi(s_0)/\sigma(s_2)}{s_2+2s_0}
     =\omega(s_2).
   $$
   It follows that $\omega$ is monotonically decreasing.
\end{proof}
A consequence of this lemma is that for the concavity of the function
$\bar \psi^2((\Theta^2)^{-1}(s)) = \bar \psi^2(\Theta^{-1}(\sqrt{s}))$ it is
thus sufficient to show that $\bar \psi \circ \Theta^{-1}$ is concave.

Finally we conjecture that a similar result to the proposition above also holds
more generally, i.e., that a sufficient condition for concavity of $g\circ\psi
\circ g^{-1}$ is the concavity of $\psi$ where $g$ belongs to a class of
suitably chosen functions.

  \section{Linear regularisation approaches}

  Our goal in this section is to draw conclusions from Corollary~\ref{cor1} for
  \emph{linear regularisation methods}. Taking into account the setting of
  Corollary~\ref{cor1} we assume throughout this section that the index function
  $\Psi(\lambda)$ is concave and strictly increasing for all $0<\lambda<\infty$
  satisfying the limit conditions $\lim \limits_{\lambda \to +0}
  \Psi(\lambda)=0,$ $\lim \limits_{\lambda \to \infty} \Psi(\lambda)=\infty$
  (cf.~Lemma~\ref{lem2}), and (\ref{limquo}). Moreover, we set   
  \begin{equation} \label{eq:ass1reg}
    \chi(\lambda):=\frac{\Psi^{-1}(\lambda)}{\lambda}\,, \quad 0<\lambda<\infty,
    \quad \mbox{and} \quad \bar \psi(t):=
    \frac{1}{\sqrt{\chi\left(\frac{1}{t}\right)}}\,, \quad 0<t<\infty \,.
  \end{equation}
  Then $\chi$ is an increasing index function with $\lim \limits_{\lambda \to
  \infty} \chi(\lambda)=\infty$ and $\bar \psi$ is an increasing index function
  with $\lim \limits_{t \to +0} \bar \psi(t)=0$, hence a rate function. As
  outlined in section~3 under these assumptions we have $\bar\psi(A^*A)[B(R)] =
  B_\chi(R)$ and the best case for regularised solutions $f_\alpha$
  approximating the exact solution $f \in X_\chi$ based on data $g^\delta$
  satisfying (\ref{noise}) by using an a priori choice $\alpha=\alpha(\delta)$
  or a posteriori choice $\alpha=\alpha(\delta,g^\delta)$ is to achieve the
  order optimal convergence rate (\ref{specialrate1}). It is a specific
  consequence of interpolation theory and can be seen easily by inspection of
  Corollary~\ref{cor1} that a successful use requires the focus on
  regularisation methods which yield regularised solutions of appropriate
  smoothness. Precisely, there must be a ball $B_\chi(R)$ to which the elements
  $f_\alpha$ belong for all $\alpha>0$ attributed to sufficiently small
  $\delta>0$ and $g^\delta$ satisfying (\ref{noise}).

  \subsection{General linear regularisation schemata}

  In a first approach we are going to consider \emph{linear regularisation
 schemes} as described in many textbooks on linear regularisation theory (see,
 e.g., \cite[Chap.~4]{EHN96}, \cite[Chap.~2]{Groe84} and
 \cite{BakuGon04,Bau87,Kirsch96,Lou89,Rieder03}). We consider approximate
 solutions
  \begin{equation}
   \label{eq:noisyregmeth}
  f_\alpha := h_\alpha(A^*A) A^* g^\delta.
  \end{equation}
  to $f$ based on a family of piecewise continuous real functions $h_\alpha(t),
  \;\, 0<t\le \|A\|^2$, to which we assign bias functions
  $$r_\alpha(t):=t\,h_\alpha(t)-1, \;\, 0<t\le \|A\|^2\,.$$ 
  These functions depend on a regularisation parameter $\alpha \in
  (0,\alpha_{max}]$, where $\alpha_{max}$ may be a finite real number or
  $\infty$. Small $\alpha>0$ characterise good approximation of the original
  problem (\ref{opeq}), whereas larger values $\alpha$ are connected with more
  stability. Hence, an appropriate trade-off between the two conflicting goals
  approximation and stability can be controlled by the choice of $\alpha.$ We
  say that such a function $h_\alpha$ describes a linear regularisation method
  if the properties
\begin{equation}  \label{v1neu}
  \lim \limits _{\alpha \to +0} \,r_\alpha(t)= 0, \qquad 0<t \le \|A\|^2,
\end{equation}
and
\begin{equation} \label{v1}
\sup \limits _{0<\alpha \le \alpha_{max}}\; \sup \limits _{0<t \le \|A\|^2}\;
t\,|\,h_\alpha(t)\,| \, \le \, C_1
\end{equation}
with a constant $C_1>0$ hold.
Because of (\ref{v1}) we have another constant $C_2>0$ such that
$$\sup \limits _{0<\alpha \le \alpha_{max}}\; \sup \limits _{0<t \le \|A\|^2}
  |\,r_\alpha(t)\,| \, \le \,C_2  $$
and hence for all $0<\alpha \le \alpha_{max}$ the estimate
$$\|Af_\alpha-g^\delta\|=\|(Ah_\alpha(A^*A)A^*-I)g^\delta\| \le \left[\sup
  \limits _{0<t \le \|A\|^2} \; |\,r_\alpha(t)\,|\right] \, \|g^\delta\| \le
  C_2\, \|g^\delta\|.$$
This implies  the limit condition
$\lim \limits_{ \alpha \to +0} \|Af_\alpha-g^\delta\|=0$ for all data $g^\delta
\in Y$. As a consequence we have that there is always a parameter choice
$\alpha=\alpha(\delta,g^\delta),\; 0<\delta \le \delta_{max},$ such that
$$\|Af_{\alpha(\delta,g^\delta)}-g^\delta\| \le C_{dis}\,\delta \qquad (0<\delta
  \le \delta_{max})$$
for some prescribed constant $C_{dis}>0$. If the mapping $\alpha \mapsto
\|Af_{\alpha}-g^\delta\|$ is even continuous, then the
discrepancy principle can be realised by a parameter choice
$\alpha_{dis}=\alpha_{dis}(\delta,g^\delta)$ satisfying the equation
(\ref{Morozov}).

Here we call a rate function $\bar \varphi$ a \emph{qualification} of the
regularisation method generated by $h_\alpha$ if there is a constant
$C_{quali}>0$ such that
\begin{equation} \label{quali}
  \sup \limits _{0<t \le \|A\|^2} \,|r_\alpha(t)| \bar \varphi(t) \,\le \,
  C_{quali} \; \bar \varphi(\alpha) , \qquad 0<\alpha \le \alpha_{max}.
\end{equation}
Now we are going to study under what conditions the inequality (\ref{eq:ratexi})
in Corollary~\ref{cor11} can be fulfilled here with $\bar \xi(\delta)=\delta$.
First we obtain
\begin{equation} \label{help1}
\|Af_\alpha-g\| = \|A r_\alpha(A^*A)f + Ah_\alpha(A^*A)A^*(g^\delta-Af)\| \le
\|A r_\alpha(A^*A)f\| +C_1 \delta.
\end{equation}
In order to apply that corollary for obtaining a convergence rate
(\ref{specialrate1}) we assume $f \in B_\chi(R_1)=\bar \psi(A^*A)[B(R_1)]$
taking into account the cross-connection (\ref{eq:ass1reg}). So let $f=\bar
\psi(A^*A)v,\; \|v\| \le R_1$. Provided that $\Theta(t):=\sqrt{t} \bar \psi(t)$
is a qualification of the method with constant $C_{quali}>0$ this gives with
(\ref{help1})
\begin{equation} \label{help2}
  \|Af_\alpha-g\| \le \left[ \sup \limits _{0<t \le \|A\|^2} \,r_\alpha(t)
  \Theta(t)\right] R_1 + C_1 \delta \le C_{quali} R_1\, \Theta(\alpha) + C_1
  \delta
\end{equation}
and hence an estimate of type (\ref{eq:ratexi}) is fulfilled with $\bar
\xi(\delta)=\delta$ when an a priori parameter choice $\alpha
=\Theta^{-1}(\delta)$ is used.

Next we will check whether $f_\alpha \in B_\chi(R_2)$ for some $0<R_2<\infty$.
We have $$f_\alpha= h_\alpha(A^*A) A^* (g^\delta-Af)+ h_\alpha(A^*A) A^*Af$$ and
after some reformulation
$$f_\alpha= \bar \psi(A^*A)\left[h_\alpha(A^*A) (\psi(A^*A))^{-1}(A^*A)^{1/2}
\tilde g + h_\alpha(A^*A) A^*Av\right]$$ 
with $\;\|\tilde g\| \le \delta$, since the different functions of $A^*A$ are
commutable. Now let the interplay of the regularisation method expressed by
$h_\alpha(t)$ and the parameter choice $\alpha=\alpha(\delta,g^\delta)$ be such
that there is a constant $C_{para}>0$ with
\begin{equation} \label{parabound}
  \sup \limits _{0<t \le \|A\|^2}
  \frac{\sqrt{t}\,|h_{\alpha(\delta,g^\delta)}(t)|\,\delta}{\bar \psi(t)}\, \le
  \, C_{para}, \qquad 0<\delta \le \delta_{max}\,.
\end{equation}
The upper bound $C_{para}$ in (\ref{parabound}) must hold for all data $g^\delta
\in Y$ associated with the noise level $\delta>0$ and satisfying (\ref{noise}),
where the case of an a priori parameter choice $\alpha=\alpha(\delta)$ should be
included as a special case. Under (\ref{parabound}) we have with (\ref{v1})
$$ \|h_{\alpha(\delta,g^\delta)}(A^*A) (\psi(A^*A))^{-1}(A^*A)^{1/2} \tilde g +
   h_{\alpha(\delta,g^\delta)}(A^*A) A^*Av\| \le R_2:=C_{para}+ C_1\,R_1,$$
in other terms $f_{\alpha(\delta,g^\delta)} \in \bar
\psi(A^*A)[B(R_2)]=B_\chi(R_2)$.

If there is a function $\Gamma(\alpha)$ satisfying for sufficiently small
$\alpha>0$ the inequality
\begin{equation} \label{para1}
  \left[\sup \limits _{0<t \le \|A\|^2} \frac{\sqrt{t}|h_\alpha(t)|}{\bar
  \psi(t)}\right]\, \le \,\Gamma(\alpha)
\end{equation}
such that
\begin{equation} \label{para2}
\delta\,\Gamma(\alpha(\delta,g^\delta))\,\le C_{para}, \qquad 0<\delta \le
\delta_{max}\,,
\end{equation}
this represents a sufficient condition for (\ref{parabound}). In particular, if
moreover the a priori parameter choice
$\alpha(\delta,g^\delta):=\Theta^{-1}(\delta)$ satisfies (\ref{para2}) we have
an estimate of type (\ref{eq:ratexi}) with $\bar \xi(\delta)=\delta$ for that a
priori parameter choice whenever $\Theta$ is a qualification of the
regularisation method under consideration.

 Hence the considerations above gave a sketch of the proof for the following
proposition as a consequence of Corollary~\ref{cor11}:

\begin{proposition} \label{pro2}
Under the standing assumptions of this section including (\ref{eq:ass1reg}) let
$f \in X_\chi=\range(\bar \psi(A^*A))$ and consider regularised solutions
(\ref{eq:noisyregmeth}) with a generator function $h_\alpha$ that determines the
regularisation method and satisfies (\ref{v1neu}) -- (\ref{v1}) as well as
(\ref{para1}) with some function $\Gamma$ such that $\Theta(t):=\sqrt{t}\bar
\psi(t)$ satisfies (\ref{para2}) with some constant $C_{para}>0$ and is a
qualification of the method (cf.~(\ref{quali})).
Then for the a priori regularisation parameter choice
$\alpha=\alpha(\delta):=\Theta^{-1}(\delta) \to +0$ as $\delta \to +0$ we have
the convergence rate
\begin{equation} \label{pro2rate}
 \|f-f_\alpha\|\,=\,\mathcal{O}\left(\delta\,\sqrt{\Psi(\bar K/\delta^2)}\right) \qquad
   \mbox{as} \qquad \delta \to +0
   \end{equation}
with some constant $\bar K>0$.
\end{proposition}

Note that in Proposition~\ref{pro2} the rate (\ref{pro2rate}) also holds for any
other parameter choice $\alpha=\alpha(\delta,g^\delta)$ that fulfils the
inequalities (\ref{para2}) and
\begin{equation} \label{ratedelta}
\|Af_{\alpha(\delta,g^\delta)}-g\| \le \bar C\,\delta, \qquad 0<\delta \le
\delta_{max},
\end{equation}
with some constant $\hat C>0$.

\begin{example} {\rm
The most prominent example of a linear regularisation method
(\ref{eq:noisyregmeth}) is the Tikhonov regularisation with the generator
function $h_\alpha(t)=\frac{1}{t+\alpha}$ and with the bias function
$r_\alpha(t)=\frac{\alpha}{t+\alpha}$, where the requirements (\ref{v1neu}) and
(\ref{v1}) are satisfied for the constants $C_1=C_2=1$. It is well known that
all concave rate functions $\bar \varphi$ are qualifications of the method
satisfying (\ref{quali}) with the constant $C_{quali}=1$. From that class we
consider the monomials $\bar \varphi(t)=t^\nu$ for exponents $0<\nu \le 1.$ Then
$\Theta(t)=\sqrt{t} \bar \psi(t)$ is a qualification with the same constant for
the Tikhonov regularisation in case of a rate function $\bar \psi(t)=t^\mu$ with
$0< \mu \le 1/2$. Taking into account (\ref{eq:ass1reg}) this rate function is
associated with $\chi(\lambda)= \lambda^{2\mu}$ and the strictly concave
function $\Psi(\lambda)=\lambda^{\frac{1}{2\mu+1}}$. By the estimate
(\ref{help2}) we have then (\ref{ratedelta}) with $\bar C=R_1+1$ for
$f=(A^*A)^\mu v, \;\|v\| \le R_1$ and for the a priori parameter choice
 \begin{equation} \label{apriH}
 \alpha=\Theta^{-1}(\delta)=\delta^{\frac{2}{2\mu+1}}\,.
 \end{equation}

To derive a function $\Gamma$ such that (\ref{para1}) is valid, we exploit the
inequality
$$ \frac{t^\kappa}{t+\alpha} \,\le \, (1-\kappa)^{1-\kappa}\kappa^\kappa\, 
    \alpha^{\kappa-1}\,,$$
which holds for all $t>0,\, \alpha>0$ and $0 < \kappa < 1$. In the limit case
$\kappa=0$ we also have the inequality $1/(t+\alpha) \le 1/\alpha.$ Thus there
is a constant $\hat c>0$ depending on $\kappa \in [0,1)$ such that $
\frac{t^\kappa}{t+\alpha} \le \frac{\hat c}{\alpha^{1-\kappa}}$. By setting
$\kappa:=1/2-\mu$ we obtain for $\bar \psi(t)=t^\mu,\;0<\mu \le 1/2$ the
inequality (\ref{para1}) with the function $$\Gamma(\alpha)=\frac{\hat
c}{\alpha^{\,\mu+\frac{1}{2}}}\,.$$ Then one easily verifies that
$\delta\,\Gamma(\Theta^{-1}(\delta)) \le \frac{\hat c
\delta}{\left(\delta^{\frac{2}{2\mu+1}}\right)^{\mu+\frac{1}{2}}}=\hat c$ and
that (\ref{para2}) is fulfilled with $C_{para}=\hat c$. Hence
Proposition~\ref{pro2} applies and we obtain for the parameter choice
(\ref{apriH}) and all $0<\mu \le 1/2$ the optimal convergence rate 
$$\|f-f_\alpha\|=\mathcal{O}\left(\delta^{\frac{2\mu}{2\mu+1}}\right) \quad
   \mbox{as} \quad \delta \to +0\,.$$
         The best possible rate obtained in that way is
  $\|f-f_\alpha\|=\mathcal{O}\left(\sqrt{\delta}\right)$ for $\mu=1/2$. For $\mu
  >1/2$ the function $\Psi$ remains strictly concave, but a finite function
  $\Gamma(\alpha)$ in (\ref{para1}) fails to exist, since we have $\sup \limits
  _{0<t \le \|A\|^2} \frac{\sqrt{t}}{\bar \psi(t)(t+\alpha)}=+\infty$.
      The limitation of Proposition~\ref{pro2} to lower H\"older rates than the
   saturation of Tikhonov's method admits seems to be a consequence of the fact
   that our approach based on Corollary~\ref{cor11} and the construction
   (\ref{eq:noisyregmeth}) do not interact good enough in case of higher
   smoothness of $f$. In order to overcome that effect, we will consider another
   approach in the following subsection.
}\end{example}

\subsection{Regularisation with unbounded operators and range inclusions}

In a second approach, under a non-standard source condition (\ref{Gsmooth})
characterising the available a priori knowledge on the solution smoothness, we
exploit a variant of the Tikhonov regularisation with regularised solutions
\begin{equation}
   \label{eq:Gregmeth}
  f_\alpha := G(GA^*AG+\alpha I)^{-1}GA^*g^\delta\,,
  \end{equation}
  where $G: X \to X$ is an injective bounded self-adjoint positive definite
linear operator $G: X \to X$ with non-closed range, i.e., zero is an
accumulation point of the spectrum $\spec(G)$ of the operator $G$. Since the
unbounded linear operator with $B=G^{-1}: \range(G) \subseteq X \to X$ is
frequently a differential operator, this approach is sometimes called
\emph{regularisation with differential operators}. Precisely, by construction
the element $f_\alpha \in \range(G)$ is well-defined for all $\alpha>0$ as the
minimiser of the extremal problem $$T_\alpha(\tilde f):= \|A \tilde
f-g^\delta\|^2+ \alpha \|B \tilde f\|^2 \to \min, \quad \mbox{subject to} \quad
\tilde f \in \range(G),$$ and then the penalty term in $T_\alpha$ contains
derivatives of the function $\tilde f$.

  To apply Corollary~\ref{cor11} under our setting (\ref{eq:ass1reg}) we assume
  $f \in G[B(R_1)]$, with $G[B(R)]$ from (\ref{Gset}), and a link condition
  \begin{equation} \label{linkGchi}
  \range(G) \subseteq X_\chi\ =\range(\bar \psi(A^*A))\,,
  \end{equation}
  which is equivalent to
  \begin{equation} \label{linkGpsi}
    \|G w \| \le C \,\|\bar \psi(A^*A) w\|, \qquad \mbox{for all} \quad w \in X,
  \end{equation}
  with some $C>0$. Then from \cite[Lemma~6.2]{HofMAt07} we obtain that $f \in
  G[B(R_1)]$ implies $f \in B_\chi(CR_1)=\bar \psi(A^*A)[B(CR_1)]$.

  Along the lines of the paper \cite{ChengYam00} by {\sc Cheng} and {\sc
  Yamamoto} we consider an a priori parameter choice $\alpha=\alpha(\delta)$ as
  \begin{equation} \label{cheng}
  \underline c \, \delta ^2 \le \alpha(\delta)\le \overline c \,\delta^2, \qquad
  0<\delta \le \delta_{max},
  \end{equation}
  with constants $0<\underline c \le \overline c<\infty$, for which we obtain
  from $T_\alpha(f_\alpha) \le T_\alpha(f)$ the inequalities
  $$\|Af_{\alpha(\delta)}-g^\delta\|^2+\alpha(\delta)
    \|G^{-1}f_{\alpha(\delta)}\|^2 \le \|Af-g^\delta\|^2+\alpha(\delta)
    \|G^{-1}f\|^2 \le \delta^2 +\overline c\,\delta^2\,R_1^2\,.$$
  Now we have
  $$\|Af_{\alpha(\delta)}-g\| \le \bar C \,\delta, \quad \mbox{with} \quad \bar
    C= \sqrt{1+\overline{c}\,R_1^2}+1 $$
  satisfying condition (\ref{eq:ratexi}) with $\xi(\delta)=\delta$  and
  $$\|G^{-1}f_{\alpha(\delta)}\| \le
    \sqrt{\frac{\delta^2}{\alpha(\delta)}+\|G^{-1}f\|^2} \le
    \sqrt{\frac{1}{\underline c}+R_1^2}=:R_2\,.$$
  This yields $f_{\alpha(\delta)} \in G[B(R_2)]$, thus $f_{\alpha(\delta)} \in
 B_\chi(CR_2)=\bar \psi(A^*A)[B(CR_2)]$ and consequently an estimate of type
 (\ref{eq:specialxi}) with $\xi(\delta)=\delta$ and $CR_1,CR_2$ instead of
 $R_1,R_2$.
  With the above considerations we have shown the convergence rate result of the
  following proposition again as a consequence of Corollary~\ref{cor11}:

\begin{proposition} \label{pro3}
Under the standing assumptions of this section including (\ref{eq:ass1reg}) let
$f$ satisfy (\ref{Gsmooth}), where the link condition (\ref{linkGchi}) is valid.
Then for the a priori regularisation parameter choice (\ref{cheng}) we have the
convergence rate (\ref{pro2rate}) with some constant $\bar K>0$.
\end{proposition}

  Due to \cite[Corollary 4.5]{HofMS08} for all concave $\Psi$ fulfilling the
  standing assumptions of this section the rate (\ref{pro2rate}) is even order
  optimal in the sense of
   $$\mathcal{O}\left(\delta\,\sqrt{\Psi(\bar K/\delta^2)}\right)=
     \mathcal{O}\left(\omega(G[B(R_1)],\delta)\right) \quad \mbox{as}
     \quad\delta \to +0\,.$$
 As already discussed in the introduction the requirement (\ref{linkGchi}) gets
 stronger for higher rates in (\ref{pro2rate}).
  In many applications (see as an illustration the examples in \cite{HofYam05})
  one can only verify range inclusions of the form (\ref{rhorange}) with some
  rate function $\bar \varrho$.
  Under operator monotonicity of the function $[\bar \rho^{-1}(\sqrt{t})]^2$
  (\ref{rhorange}) implies (\ref{linkGchi}) with $\bar \psi(t)=\bar
  \rho^{-1}(\sqrt{t})$ and 
  $\chi(\lambda)=\left[\bar \rho^{-1}\left(1/\sqrt{\lambda}\right)\right]^{-2}$.

\smallskip

 In order to verify in general for what index functions $\chi$ a range inclusion
 (\ref{linkGchi}) with $\range(G) \subseteq X_\chi$ is fulfilled, one can use
 the \emph{spectral theorem} for unbounded self-adjoint operators $T$ (see
 \cite[Chapter VII.3]{Wer00} and also \cite[Chapter VIII]{ReedSimon80}).
 In the Hilbert space~$X$, the injective, densely defined, self-adjoint,
 positive definite, and unbounded linear operator $T$ is unitarily invariant to
 a multiplication operator $\mathcal{M}$ expressed by a real multiplier function
 $m$. This means that there are a measure space~$(\Sigma,\mathcal A,\mu)$ with
 finite measure $\mu$, a unitary operator~$\mathcal{U}\colon X \to
 L^2(\Sigma,\mathcal A,\mu)$ and a real measurable function $m(t),\;t \in
 \Sigma$, such that $[\mathcal{M}h](t):=m(t)h(t)$ a.e., where $\mathcal{M}$ maps
 in $L^2(\Sigma,\mathcal A,\mu)$, and
$$ \mathcal{U}\, T\, \mathcal{U}^*\,h = \mathcal{M}\,h = m \cdot h$$
for all $h$ from the domain of $\mathcal{M}$. We note that the closure of the
range $\range(m)$ of the multiplier function $m$ and the spectrum $\spec(T)
\subseteq [\|T\|^{-1},\infty)$ of the operator $T$, possessing $+\infty$ as an
accumulation point, coincide. Moreover, we have for index functions $\psi \in
\mathcal{I}$ and $h$ from the domain of $\psi(\mathcal{M})$ $$ \mathcal{U}\,
\psi(T)\, \mathcal{U}^*\,h = \psi(M)\,h=\psi(m) \cdot h\,. $$ Then by using the
notations $\hat f:=\mathcal{U}f \in L^2(\Sigma,\mathcal A,\mu)$ and
$\widehat{(Gw)} := \mathcal{U}\,G\,w \in L^2(\Sigma,\mathcal A,\mu)$ by
definition we immediately find that $\range(G) \subseteq X_\chi$ is equivalent
to the condition that
\begin{equation} \label{incco}
 (Gw,\chi(T)Gw)=
 (\widehat{(Gw)},\chi(\mathcal{M})\widehat{(Gw)})_{\scriptscriptstyle
 L^2(\Sigma,\mathcal A,\mu)}=\int \limits_\Sigma \chi(m(t))|\widehat{(Gw)}(t)|^2
 dt <\infty
\end{equation}
holds for all $w \in X.$ In Example~\ref{eximaging} with background in imaging
(cf.~\cite{Scherzetal09}) we will consider the special case that $\mathcal{U}$
denotes the two-dimensional Fourier transform and that the corresponding measure
space is $(\R^2, \mathcal{B}(\R^2),\mu)$ with the associated Borel
$\sigma$-algebra and measure. In that example, $T$ and $G$ are commuting
operators, both non-compact with a non-closed range.

On the other hand, in Example~\ref{HSex} we will exploit the one-dimensional Fourier transform to formulate sufficient conditions such that classical source conditions are
satisfied for linear compact integral operators.

  \section{Examples}

  In the remaining examples we illustrate the theory. All the
  occurring operators $A$ are linear integral operators. First the
  Example~\ref{eximaging} refers to convolution operators $A$ which
  occur, for example, when the deblurring of noisy images is under
  consideration. Then the Example~\ref{HSex} illustrates the \emph{low
  rate} case where an integral equation with a smooth kernel is solved
  and it is known that the solution is in a Sobolev space. The
  situation here is similar as in the case of elliptic partial
  differential equations and has been discussed in~\cite{Boetal06}. In
  contrast to the PDE situation here convergence rates are low,
  typically of the form $O(|\log(\delta)|^{-k})$. The final
  Example~\ref{exfinal} illustrates the \emph{high rate} case where a
  derivative of data in the range of an integral operator with smooth
  kernel is considered. The high convergence rates are here of the
  form $O(\delta |\log(\delta)|^k)$.

  In the examples we consider functions over $\R^d\;(d=1,2)$ and
  Sobolev spaces $H^l(\R^d)\;(l=1,2,...)$ of Hilbert type will be used
  with norms $\|\cdot\|_l$ defined by
  $$
    \|x\|_l^2 = \frac{1}{(2\pi)^{d}} \int_{\R^d}
    (1+|\omega|^2+\cdots+|\omega|^{2l}) |\hat{x}|^2\,d\omega,
  $$
  where $\hat{x}=\hat{x}(\omega),\;\omega \in \R^d$, is the Fourier
  transform of $x$. Now let $E_l:H^l(\R^d)\rightarrow L^2(\R^d)$
  denote the embedding and $E_l^*$ the adjoint of $E_l$. Then $E_l
  E_l^* : L^2(\R^d) \rightarrow L^2(\R^d)$ is an integral operator and
  $$
    \widehat{E_l E_l^* y}(\omega) =
    \frac{\hat{y}(\omega)}{1+|\omega|^2+\cdots+ |\omega|^{2l}}.
  $$

  \begin{example} \label{eximaging}
  {\rm In this example with $X=Y=L^2(\R^2)$ we are interested in deblurring,
  that means in finding a true picture which is characterised by a function $f
  =f(t)\in L^2(\R^2),\;t=(t_1,t_2)^{\scriptscriptstyle T},$ that satisfies a
  linear operator equation (\ref{opeq}) of convolution type
  \begin{equation} \label{convolution}
    Af\,(s) = \int \limits_{\R^2} k(s-t) f(t)\, dt = g(s), \quad
  s=(s_1,s_2)^{\scriptscriptstyle T}\in \R^2,
  \end{equation}
  where $g \in L^2(\R^2)$ is a blurred image of $f$ which is additionally
  contaminated with noise such that only the noisy blurred image $g^\delta \in
  L^2(\R^2)$ satisfying (\ref{noise}) available as data. Following
  \cite[Chapter~3]{BeBo98} the kernel function
  $k(\tau),\;\tau=(\tau_1,\tau_2)^{\scriptscriptstyle T} \in \R^2,$ is called
  \emph{point spread function} of a space invariant imaging system under
  consideration. We assume that the kernel is such that its Fourier transform
  $\hat k=\hat k(\omega)\;,\omega=(\omega_1,\omega_2)^{\scriptscriptstyle T}$,
  called \emph{transfer function} is bounded. Different variants of such
  deblurring problems are presented and analysed in \cite{BeBo98}. As a
  reference situation we exploit for illustration a variant of an out-of-focus
  blur for which
  $$
    \hat k(\omega) = 2 \frac{J_1(D|\omega|)}{D|\omega|}
  $$
  where $J_1$ is the Bessel function of order one and $D$ is the radius of the
  circle of confusion (cf.~\cite[formula (3.25) on p.60]{BeBo98}). The linear
  convolution operator $A:L^2(\R^2) \to L^2(\R^2)$ in this example has a
  non-closed range but it is non-compact and the kernel is not square
  integrable. 

  In order to apply our theory to this example one needs to find an index
  function $\theta$ and a symmetric positive definite operator $T$ such that
  $A^*A = \theta(T)$. A natural choice in this context is $T= -\Delta$ and in
  this case $\theta$ needs to satisfy $|\hat k(\omega)|^2 = \theta(|\omega|^2)$.
  This, however, is not possible, as $\hat k(\omega)$ is zero for some finite
  $\omega$ but an index function has to satisfy $\theta(\lambda) > 0$ for all
  $\lambda > 0$ and it can only be zero asymptotically at zero or infinity. It
  is thus not possible to get error bounds for the deblurring problem using the
  variable Hilbert scale theory and $T=-\Delta$.

  One does not have this problem if one chooses $T=(A^*A)^{-1}$. Let us define
  the solution smoothness as $f \in H^l(\R^2)$. Then we have the operator
  $G=E_l^*E_l$ in (\ref{Gsmooth}) characterising the associated non-standard
  source condition. To find index functions $\chi$ that satisfy the link
  condition (\ref{linkGchi}) we can make use of formula (\ref{incco}) taking
  into account that $m(\omega)=1/|\hat k(\omega)|^2$ and
  $$
    |\widehat{Gw}(\omega)|^2=(1+|\omega|^2+\cdots+|\omega|^{2l})^{-1}| \hat w
    (\omega)|^2.
  $$
  Then the range inclusion $\range(G) \subseteq X_\chi$ takes the form 
  \begin{equation} \label{ex2cond}
   \chi\left(\frac{1}{|\hat k(\omega)|^2}\right)
   \frac{1}{(1+|\omega|^2+\cdots+|\omega|^{2l})} \le \bar C < \infty \quad
   \mbox{for all} \quad \omega \in \R^2\,.
  \end{equation}
  This range condition can only be satisfied if $\chi$ is bounded, i.e., i.e.
  $\chi(\lambda) \leq C < \infty$ for all $\lambda > 0$. This is again a
  consequence of the existence of zeros of $\hat k(\omega)$ for finite $\omega$.
  A consequence of the finiteness of $\chi$ is $L^2(\R^2) \subset H_\chi $ and
  it follows that the ``source condition" reduces to $f\in L^2(\R)$ which does
  not lead to an error bound.

  The failure of the above attempts to get error bounds clearly illustrates the
  need to extend the variable Hilbert scale theory to be able to cope with the
  deblurring problem. One can, however, deal with a partial deblurring problem.
  Observe that one has the asymptotics
  $$
    2 \left|\frac{J_1(D|\omega|)}{D|\omega|}\right| \asymp  |\omega|^{-3/2}
  $$
  for large $|\omega|$  (cf.~\cite[formula (3.29) on p.60]{BeBo98}). It follows
  that $\hat k (\omega) = |\omega|^{-3/2} \kappa(\omega)$ for some bounded
  $\kappa(\omega)$. The first factor $|\omega|^{3/2}$ relates to a ``smoothing
  component" of the out-of-focus blur situation. We now consider inversion of
  this smoothing component only. For this we introduce an integral operator $A$
  with kernel $k$ which satisfies
  \begin{equation} \label{outoffocus}
      \hat k(\omega) = |\omega|^{3/2}.
  \end{equation}
  For the ``partial" out-of-focus blur situation (\ref{outoffocus}) and monomials
  $\chi(\lambda)=\lambda^\kappa,\;\kappa>0$, we have (\ref{ex2cond}) if and only
  if $\kappa \le \frac{2l}{3}.$ With the relation
  $\chi(\lambda)=\Psi^{-1}(\lambda)/\lambda$ this corresponds with
  $\Psi(\lambda)\le \lambda^{3/(2l+3)}$. Hence based on Proposition~\ref{pro3}
  for the situation (\ref{outoffocus}) and under $f \in H^l(\R^2)$ a best
  possible convergence rate
  $$ \|f-f_\alpha\|_{\scriptscriptstyle
     L^2(\R^2)}=\mathcal{O}\left(\delta^{\frac{2l}{2l+3}}\right) \quad \mbox{as}
     \quad \delta \to +0$$
  can be obtained by Tikhonov regularisation with $H^l$-penalty term.
  }\end{example}

  \begin{example} \label{HSex}
  {\rm In this example we consider compact forward operators $A$ in equation
  (\ref{opeq}) with $X=Y=L^2(\R)$ in form of linear operators $A: L^2(\R) \to
  L^2(\R)$, for which the range of the operator $K:=\bar \psi(A^*A)$ is a subset
  of $H_\phi$ with some index function $\phi$ and some rate function $\bar
  \psi$. That means, we have $\range(K) \subseteq H_\phi$ and a classical source
  conditions (\ref{scclassic}) is valid for $f \in H_\phi$ implying the
  corresponding convergence rates in regularisation. In this context, let $K$ be
  a linear Fredholm integral operator of Hilbert-Schmidt type. For such
  operators one can provide conditions on the kernel which guarantee this range
  condition.
  \begin{lemma}
    \label{lem3-ex}
    Let $K:L^2(\R)\rightarrow L^2(\R)$ be a Hilbert-Schmidt operator with kernel
    $k(t,s)\in L^2(\R^2)$. Furthermore, let $\tilde{K}:L^2(\R)\rightarrow
    L^2(\R)$ be an integral operator with kernel $\tilde{k}(\omega,s)=\int_R
    e^{-i\omega t} k(t,s)\, dt$. Then $\tilde{K}$ is a Hilbert-Schmidt operator
    and
    $$
       \tilde{K} x = \widehat{Kx}, \quad x\in L^2(\R).
    $$
  \end{lemma}
  \begin{proof}
    The adjoint operator $K^*$ of $K$ is an integral operator with kernel
    $k^*(s,t)=\overline{k(t,s)}$ as a consequence of the theorem of Fubini.
    By Plancherel's theorem one has
    $$
      K^* u = \frac{1}{2\pi} \tilde{K}^* \hat{u}.
    $$
    An application of Parseval's identity several times gives for
    $u,v\in L^2(\R)$:
    \begin{align*}
      \frac{1}{2\pi} (\hat{u}, \widehat{Kv}) & = (u, Kv) \\
      & = (K^* u, v) \\
      & = \frac{1}{2\pi} (\tilde{K}^* \hat{u}, v) \\
      & = \frac{1}{2\pi}(\hat{u}, \tilde{K} v).
    \end{align*}
  \end{proof}

  \begin{proposition}
    Let $K:L^2(\R)\rightarrow L^2(\R)$ be a Hilbert-Schmidt operator where the
    Fourier transform $\tilde{k}(\omega,s)=\int_\R e^{-i\omega t} k(t,s)\, dt$
    of the kernel of $K$ satisfies
    $$
      \int_\R \int_\R \phi(\omega^2) |\tilde{k}(\omega,s)|^2\, ds\,
       d\omega < \infty
    $$
    for some index function $\phi$. Then $\range(K) \subseteq H_\phi$.
  \end{proposition}
  \begin{proof}
    By Lemma~\ref{lem3-ex} one has
    $$
      \hat{y}(\omega) = \int_\R \tilde{k}(\omega,s) x(s)\, ds
    $$
    which we insert into the following bound, obtained from the Cauchy-Schwarz
    inequality:
    $$
      \left|\int_\R \tilde{k}(\omega,s) x(s)\, ds \right|^2 \leq
      \int_\R |\tilde{k}(\omega,s)|^2\, ds \, \|x\|^2.
    $$
    It follows that for $y=Kx$ with $x\in L^2(\R)$ one has
    \begin{align*}
      \|y\|_\phi^2 &= \frac{1}{2\pi} \int_\R \phi(\omega^2) |\hat{y}(\omega)|^2\,
      d\omega \\
      &\leq \frac{1}{2\pi}\int_\R \int_\R \phi(\omega^2) |\tilde{k}(\omega,s)|^2
      \, ds \, d\omega \, \|x\|^2
    \end{align*}
    and consequently $y\in H_\phi$.
  \end{proof}

}
\end{example}

\begin{example} \label{exfinal}
{\rm
  As a concrete application example we consider a problem from
  derivative spectroscopy~\cite{StauS68}. Here numerical derivatives
  are used to enhance the resolution of measured spectra in order to
  separate close peaks. An instance is the Eddington correction
  formula. The approach determines
  $$
    f = Lg := g - \frac{g^{(2)}}{2}
  $$
  from observed $g_\delta$ where $g^{(2)}$ is the second derivative of
  $g$. We now apply the theory developed so far to determine how well
  $f=Lg$ can be determined from spectral data $g_\delta$.

  For $f\in H^2(\R)$ and $f=Lg$ the Fourier transforms $\hat{f}$ and
  $\hat{g}$ satisfy
  $$
    \hat{f}(\omega) = (1+\omega^2/2)\, \hat{g}(\omega), \quad
    \text{a.e.}
  $$
  Using Plancherel's theorem, one obtains from this the bounds
  $$
    \frac{1}{2} \|f\|_2 \leq \|Lf\| \leq \|f\|_2, \quad
    f\in H^2(\R)
  $$
  which means in particular that $\|Lf\|$ is an equivalent norm for
  $H^2(\R)$. Using standard arguments, one can then show that $L:
  H^2(\R)\rightarrow L^2(\R)$ is a Hilbert space isomorphism. Using
  the convolution theorem one sees that $A=E_2 L^{-1} : L^2(\R)
  \rightarrow L^2(\R)$ is an integral operator with
  $$
     Af\, (t) = \frac{1}{\sqrt{2}} \int_\R \exp(-\sqrt{2}|t-s|) f(s)\,
    ds \quad t \in \R
  $$
  where $E_2$ denotes the embedding $H^2(\R)\rightarrow L^2(\R)$. As
  $L^{-1}$ maps $L^2(\R)$ onto $H^2(\R)$ the range of $A$ can be
  identified with $H^2(\R)$.

  In addition to the Sobolev spaces, which form a classical Hilbert
  scale, we will use a \emph{variable Hilbert scale} $H_\phi$ with
  norms $\|\cdot\|_\phi$ defined by
  $$
    \|x\|_\phi^2 = \frac{1}{(2\pi)} \int_{-\infty}^\infty
    \phi(\omega^2) \, |\hat{x}(\omega)|^2 \, d\omega
  $$
  where $\phi$ are \emph{index functions}. Note that we have here
  $H_\phi=X_\phi$ for $X=L^2(\R)$ (see~\cite{Heg92}) and the
  generating operator is the second order differential operator
  $T=-d^2/dt^2$. The index functions
  $$
    \nu_k(\lambda) = 1 + \lambda + \cdots + \lambda^k
    = \frac{\lambda^k - 1}{\lambda - 1}
  $$
  define the Sobolev spaces, in particular, one has
  $H_{\nu_k}=H^{k}(\R)$ and furthermore, the Sobolev norm is equal to
  the norm of the corresponding variable Hilbert scale:
  $$
    \|f\|_k = \|f\|_{\nu_k}, \quad f\in H_{\nu_k}.
  $$

  In this framework, we now get error bounds analogue to the ones in
  Corollary~\ref{cor1} which are again a consequence of
  Lemma~\ref{lem1}.
  \begin{corollary}\label{cor1a}
    Let $H_\nu$ be the Hilbert scales generated by $T=-d^2/dt^2$ from
    $L^2(\R)$. Furthermore, let $A: L^2(\R)\rightarrow L^2(\R)$ be a
    (convolution) operator satisfying
    $$
       A^* A = \theta(T)
    $$
    for some bounded index function $\theta$. Moreover, let $\phi,\psi$
    and $\Psi$ be index functions and $\Psi$ be concave such that
    $$
       \phi(\lambda) \le \Psi(\psi(\lambda)), \quad \phi(\lambda) \theta(\lambda)=1,\quad \mbox{for} \quad t > 0.
    $$
    If $Af=g\in H_\psi$ and if $f_\alpha$ is such that $Af_\alpha\in H_\psi$
    and
    \begin{align*}
      \|Af_\alpha - g \|_\psi & = \zeta \\
      \|Af_\alpha - g \| & = \epsilon
    \end{align*}
    then
    $$
      \|f-f_\alpha\| \leq \epsilon \sqrt{\Psi(\zeta^2/\epsilon^2)}.
    $$
  \end{corollary}
  \begin{proof}
    Note that $\theta(T)$ is well defined by the Fourier transform and
    as $\theta$ is bounded, so is $A$. Furthermore it follows from
    the condition $A^*A = \theta(T)$ that
    $\|h\|_\theta = \|Ah\|$ for all $h\in L^2(\R)$ and
    $\|h\|_{\psi\theta} = \|Ah\|_\psi$, for all $h$ with $Ah\in H_\psi$.

    By the variable Hilbert scale interpolation inequality
    (Lemma~\ref{lem1}) one has
    $$
      \|f-f_\alpha\|^2_{\phi\theta} \leq \|f-f_\alpha\|^2_\theta\,
      \Psi\left(\frac{\|f-f_\alpha\|^2_{\psi\theta}}{\|f-f_\alpha\|^2_\theta}\right).
    $$
    Now $\phi(\lambda)\theta(\lambda)=1$ and $\|f_\alpha-f\|_\theta
    = \|Af_\alpha-g\|$
    as $f_\alpha-f\in L^2(\R)$. Furthermore, $\|f_\alpha-f\|_{\psi\theta}
    = \|Af_\alpha - g\|_\psi$.
    By Lemma~\ref{concave1} one then has
    $$
      \|f_\alpha-f\| \leq \|Af_\alpha - g\|
      \sqrt{\Psi(\|Af_\alpha-g\|_\psi^2/\|Af_\alpha-g\|^2)}.
    $$
    The bound follows by another application of Lemma~\ref{concave1}.
  \end{proof}
  In comparison with Corollary~\ref{cor1} this corollary uses an
  operator $T$ which is not necessarily equal to $(A^*A)^{-1}$ but more
  importantly, the source condition is here not given as a property of
  the solution $f$ but of the data $g$.

  For the application of this corollary to the case of the Eddington
  correction formula one chooses $\theta(\lambda)=1/(1+\lambda/2)$ and
  so $\phi(\lambda) = 1 + \lambda/2$.

  In contrast to the usual case, where the source condition is stated
  as a condition on $f$, here the source condition is stated as a
  condition on (the original spectrum) $g$. This source condition
  results from physical models for the spectrum, and, in particular
  for the so-called spectral broadening. A variety of models are used,
  the most common ones are the Gaussian, Lorenz and Voigt spectra
  where a Voigt spectrum is a combination of a Lorenz and a Gaussian
  spectrum. Here we consider Gaussian spectra defined by
  $$
    g(t) = \frac{1}{\sqrt{2\pi}} \int_\R \exp(-(t-s)^2/2) v(s)\, ds
  $$
  for some $v\in L^2(\R)$. For a different discussion and more
  background on the problem, the reader may consult the paper by
  Hegland~\cite{Heg09x}.

  It follows that $g\in H_\psi$ with $\psi(\lambda) = \exp(\lambda)$.
  The concave
  function $\Psi$ can then be chosen as
  $$
    \Psi(\lambda) =
    \begin{cases}
        \lambda,   & \text{for $\lambda \leq 1$} \\
        (1 + \log(\lambda)/2)^2, &  \text{for $\lambda \geq 1$}.
    \end{cases}
  $$
  It follows that $\Psi$ is concave and that $\phi(\lambda) \le
  \Psi(\psi(\lambda)) $. As a consequence one gets the error bounds
   $$
    \|f-f_\alpha\| \leq \delta (1 + \log(\eta/\delta))
   $$
  for $\delta < \eta$ and $\|f-f_\alpha\| \leq \eta$ if $\delta \geq \epsilon$.
  The stabilisation guarantees that even if the errors are very large, the
  error of the approximation does not grow to infinity. In fact, the solution
  $f_\alpha = 0$ would probably be a good choice for the large data error case.

}\end{example}


\def\cprime{$'$}
\providecommand{\bysame}{\leavevmode\hbox to3em{\hrulefill}\thinspace}
\providecommand{\MR}{\relax\ifhmode\unskip\space\fi MR }
\providecommand{\MRhref}[2]{%
  \href{http://www.ams.org/mathscinet-getitem?mr=#1}{#2}
}
\providecommand{\href}[2]{#2}

\end{document}